\documentclass[11pt,a4paper]{article}

\usepackage{epsf,epsfig,amsfonts,amsgen,amsmath,amstext,amsbsy,amsopn,amsthm,cases,listings,color
}
\usepackage{ebezier,eepic}
\usepackage{color}
\usepackage{multirow}
\usepackage{epstopdf}
\usepackage{graphicx}   
\usepackage{pgf,tikz}
\usepackage{mathrsfs}
\usepackage[marginal]{footmisc}
\usepackage{enumerate}
\usepackage[titletoc]{appendix}
\usepackage{booktabs}
\usepackage{url}
\usepackage{mathtools}

\usepackage{pgfplots}
\usepackage{authblk}
\usepackage{amssymb}

\usepackage{wasysym}

\usepackage{empheq}

\usepackage{dsfont}

\usepackage{tikz}
\usepackage{longtable}
\usepackage{enumitem}

\usepackage{subfigure}
\pgfplotsset{compat=1.18}
\usepackage{mathrsfs}
\usepackage{wasysym} 
\usetikzlibrary{arrows}
\usepackage{aligned-overset}
\usepackage{bm}
\usepackage{bbm}
\usepackage[T1]{fontenc}
\usepackage[backref=page]{hyperref}
\usepackage[notref,notcite,color]{showkeys}

\usepackage[thinc]{esdiff}
\newcommand{\at}[2][]{#1|_{#2}}

\allowdisplaybreaks[1]

\definecolor{uuuuuu}{rgb}{0.27,0.27,0.27}
\definecolor{sqsqsq}{rgb}{0.1255,0.1255,0.1255}

\newtheorem{definition}{Definition} [section]
\newtheorem{theorem}[definition]{Theorem}
\newtheorem{lemma}[definition]{Lemma}

\newtheorem{claim}[definition]{Claim}
\newtheorem{problem}[definition]{Problem}

\newtheorem{fact}[definition]{Fact}

\begin{document}
\title{\bf\Large On a refinement of the Ahlswede--Katona Theorem}
\date{\today}
\author[1]{Jianfeng Hou\thanks{Research was supported by National Key R\&D Program of China (Grant No. 2023YFA1010202), and by the Central Guidance on Local Science and Technology Development Fund of Fujian Province (Grant No. 2023L3003). Email: \texttt{jfhou@fzu.edu.cn}}}
\author[2]{Xizhi Liu\thanks{Research was supported by the Excellent Young Talents Program
(Overseas) of the National Natural Science Foundation of China. Email: \texttt{liuxizhi@ustc.edu.cn}}}
\author[1]{Yixiao Zhang\thanks{Research was supported by National Key R\&D Program of China (Grant No. 2023YFA1010202). Email: \texttt{fzuzyx@gmail.com}}}
\affil[1]{Center for Discrete Mathematics, Fuzhou University, Fuzhou, 350108, China}
\affil[2]{School of Mathematical Sciences, 
            USTC,
            Hefei, 230026, China}
\maketitle
\begin{abstract}
    A classical theorem of Ahlswede and Katona~\cite{AK78} determines the maximum density of the $2$-edge star in a graph with a given edge density. Motivated by its application in hypergraph Tur\'{a}n problems, we establish a refinement of their result under the additional assumption that the graph contains a large independent set in which every vertex has high degree. 
    
\medskip


\medskip

\end{abstract}

\section{Introduction}\label{Sec:Introduction}
\subsection{Definitions and main results}
We identify a graph $F$ with its edge set, and use $V(F)$ to denote its vertex set. 
The number of edges in $F$ is denoted by $|F|$. 
Given integers $n$ and $m$ with $m \le \binom{n}{2}$, let $\mathcal{G}(n,m)$ be the collection of all graphs on $n$ vertices with exactly $m$ edges. 
For graphs $F$ and $G$, let $N(F,G)$ denote the number of subgraphs of $G$ that are isomorphic to $F$. 
A classical problem in extremal graph theory is to determine the maximum possible value of $N(F,G)$ when $G \in \mathcal{G}(n,m)$. 
One of the earliest result on this topic is the seminal Kruskal--Katona theorem~\cite{Kru63,Kat68}, which determines the maximum number of $K_{\ell}$, the complete graph on $\ell$ vertices, in a graph with a prescribed number of vertices and edges. 

In this work, we focus on the case $F = S_{2}$, the star with two edges (also known as a cherry). 
We first introduce two classes of graphs. 
Let $n$ and $m$ be integers satisfying $m \le \binom{n}{2}$. 
Write $m = \binom{k}{2} + \ell$, where $k$ and $\ell$ are nonnegative integers with $\ell \le k-1$. 
The \emph{quasi-clique} $C(n,m)$ is the graph on vertex set $[n]$ with edge set 
\begin{align*}
    \binom{[k]}{2} \cup \big\{ \{i, k+1\} \colon i \in [\ell] \big\}. 
\end{align*}
The \emph{quasi-star} $S(n,m)$ is the graph on the vertex set $[n]$, defined as the complement of $C\left(n,\binom{n}{2}-m\right)$. 

The classical theorem of Ahlswede--Katona~\cite{AK78} is stated as follows. 

\begin{theorem}[Ahlswede--Katona~\cite{AK78}]\label{THM:AK78S2}
    Let $n$ and $m$ be integers satisfying $m \le \binom{n}{2}$. 
    Suppose that $G \in \mathcal{G}(n,m)$. 
    Then 
    \begin{align*}
        N(S_2, G)
        \le \max\big\{N\big(S_2, S(n,m)\big),~N\big(S_2, C(n,m)\big)\big\}. 
    \end{align*}
\end{theorem}

The extension of the Ahlswede--Katona Theorem to $S_{\ell}$ with $\ell \ge 3$, the star with $\ell$ edges was achieved asymptotically only recently by Reiher--Wagner~\cite{RW18} (see also Cairncross--Mubayi~\cite{CM24} for related remarks). 

In this work, we focus on the asymptotic behavior of $N(S_2, G)$. 
Given a graph $G$ on $n$ vertices, we define its \emph{edge density} $\rho(G)$ and its \emph{$S_{2}$-density} (or \emph{cherry density}) $\rho(S_2, G)$ by 
\begin{align*}
    \rho(G)
    \coloneqq \frac{|G|}{\binom{n}{2}}
    \quad\text{and}\quad 
    \rho(S_2, G)
    \coloneqq \frac{N(S_2, G)}{3\binom{n}{3}}. 
\end{align*}
The asymptotic version of Theorem~\ref{THM:AK78S2} is as follows. 
Suppose that $G$ is a graph on $n$ vertices with edge density $x \in [0,1]$. Then 
\begin{align*}
    \rho(S_2, G)
    & \le \max\big\{ 2x -1 +(1-x)^{3/2},~x^{3/2} \big\} +o(1) \\[0.4em]
    & = o(1) + \begin{cases}
        2x -1 +(1-x)^{3/2}, & \quad\text{if}\quad x\in \left[0, \tfrac{1}{2}\right], \\
        x^{3/2}, & \quad\text{if}\quad x\in \left[\tfrac{1}{2}, 1\right]. 
    \end{cases}
\end{align*}
Note that for $x\in \left[\tfrac{1}{2}, 1\right]$, the bound $x^{3/2}$ is asymptotically attained by the quasi-clique construction. 
Observe that this construction does not contain a large independent set in which every vertex has a large degree. 
This motivates us to study the following refinement of Theorem~\ref{THM:AK78S2}. 

Let $G$ be a graph. 
We use $d_{G}(v)$ to denote the degree of a vertex $v$ in $G$. 
For every set $S \subseteq V(G)$, let
\begin{align*}
    \delta_{G}(S)
    \coloneqq \min\left\{ d_{G}(v) \colon v \in S \right\}. 
\end{align*}

Let $n, m, \ell, k$ be nonnegative integers with $m \le \binom{n}{2}$ and $k \ell \le m$. 
Define $\mathcal{G}(n,m,\ell, k)$ as the family of graphs $G \in \mathcal{G}(n,m)$ containing an independent set $I$ of size $\ell$ with $\delta_{G}(I) \ge k$. 
Equivalently, 
\begin{multline*}
    \mathcal{G}(n,m,\ell, k) \\
    \coloneqq \left\{ G \in \mathcal{G}(n,m) \colon \text{$G$ contains an independent set $I$ with $|I| \ge \ell$ and $\delta_{G}(I) \ge k$}\right\}.
\end{multline*}

\begin{problem}\label{Problem:this-note-counting-star}
    Let $n, m, \ell, k$ be nonnegative integers with $m \le \binom{n}{2}$ and $k \ell \le m$.  Determine 
    \begin{align*}
        \max\left\{ N(S_2, G) \colon G \in \mathcal{G}(n,m,\ell, k) \right\}. 
    \end{align*}
\end{problem}
Another motivation for considering Problem~\ref{Problem:this-note-counting-star} comes from its relevance to the study of hypergraph Tur\'{a}n problems in the $\ell_2$-norm, such as~\cite{HLZ25Fano}. 

Since we are interested in the asymptotic behavior of the solution to Problem~\ref{Problem:this-note-counting-star},  for real numbers $\rho, \alpha, \beta \in [0,1]$, define 
\begin{align*}
    I(S_2, \rho, \alpha, \beta)
    \coloneqq \limsup_{n \to \infty} \left\{ \rho(S_2, G) \colon G \in \mathcal{G}\big(n,\rho \tbinom{n}{2}, \alpha n, \beta n \big) \right\}. 
\end{align*}

Write $m = k \ell + \binom{a}{2} + b$, where $a$ and $b$ are nonnegative integers with $b \in [0, a-1]$. 
Let $G_{1}(n,m,\ell, k)$ be the graph with vertex set $[n]$ obtained as follows: 
\begin{enumerate}[label=(\roman*)]
    \item place a complete graph on vertex set $[a]$, 
    \item place a $k \times \ell$ complete bipartite graph with parts $[k]$ and $[n-\ell+1, n]$, 
    \item connect the vertex $a+1$ to all vertices in $[b]$. 
\end{enumerate}

Write $m = a\ell + \binom{a}{2} + b$, where $a$ and $b$ are integers with $b \in [0, a+\ell-1]$. 
Let $G_{2}(n,m,\ell, k)$ be the graph with vertex set $[n]$ obtained as follows: 
\begin{enumerate}[label=(\roman*)]
    \item place a complete graph on vertex set $[a]$, 
    \item place an $a \times \ell$ complete bipartite graph with parts $[a]$ and $[a+1, a+\ell]$, 
    \item connect the vertex $a+\ell+1$ to all vertices in $[b]$. 
\end{enumerate}

When $m \ge \binom{k}{2} + k(n-k)$, the quasi-star $S(n,m)$ contains an independent set $I$ of size at least $\ell$ in which every vertex has degree at least $k$. 
Similarly, for $m \ge k \ell + \binom{k}{2}$, it is straightforward to verify that both $G_{1}(n,m,\ell,k)$ and $G_{2}(n,m,\ell,k)$ contain an independent set of size $\ell$ where each vertex has degree at least $k$. 
Therefore, the constructions $S(n,m)$, $G_{1}(n,m,\ell,k)$, and $G_{2}(n,m,\ell,k)$ provide the following lower bounds for $I(S_2,\rho,\alpha,\beta)$, respectively. 
\begin{fact}\label{FACT:lower-bound}
    Let $\rho, \alpha, \beta \in [0,1]$ be real numbers with $\alpha + \beta \le 1$ and $\rho \ge 2 \beta -\beta^2$. 
    Then 
    \begin{multline*}
        I(S_2,\rho,\alpha,\beta) \\
        \ge \max\left\{ 2\rho -1 +(1-\rho)^{3/2},~\alpha^2\beta + \beta^2 \alpha + \rho \sqrt{\rho-2\alpha \beta},~\alpha^3+(\rho-\alpha^2)\sqrt{\rho+\alpha^2} \right\}. 
    \end{multline*}
\end{fact}

The main results of this work are as follows.
\begin{theorem}\label{Thm:Max-S2-alpha-alpha}
    Suppose that $\rho \in \left[\frac{17}{25}, \frac{7}{10} \right]$ and $\alpha \in \left[\frac{17}{100}, \frac{23}{100}\right]$. 
    Then
    \begin{align*}
        I(S_2,\rho, \alpha, \alpha) 
        = \alpha^3+(\rho-\alpha^2) \sqrt{\rho+\alpha^2}. 
    \end{align*}
\end{theorem}

\begin{theorem}\label{Thm:Max-S2-alpha-beta}
    Suppose that $\rho \in \left[\frac{17}{25}, \frac{7}{10}\right]$ and $\alpha \in \left[\frac{1}{3}, \frac{2}{5}\right]$. 
    Then 
    \begin{align*}
        I\left(S_2,\rho,\alpha, \tfrac{1}{5} \right) 
        = \max\left\{2\rho -1 +(1-\rho)^{3/2},~\alpha^3+(\rho-\alpha^2) \sqrt{\rho+\alpha^2}\right\}. 
    \end{align*}
\end{theorem}

\subsection{Counting cherries in bipartite graphs}\label{Sec:note-bipartite}
In this subsection, we consider the problem of maximizing the number of copies of $S_{2}$ in an $r \times s$ bipartite graph with a given number of edges.

Let $r, s, m$ be nonnegative integers with $m \le rs$ and $r \ge s$.
Write $m = r p + q$, where $p$ and $q$ are nonnegative integers with $q \in [0, r-1]$.
Define $B(r,s,m)$ as the bipartite graph with two parts $U = \{u_1, \ldots, u_{r} \}$ and $W = \{w_1, \ldots, w_{s} \}$ and with edge set
\begin{align*}
    \big\{ \{u_i, w_j\} \colon (i,j) \in [r] \times [p] \big\} 
    \cup \big\{ \{u_i, w_{p+1}\} \colon i \in [q] \big\}. 
\end{align*}

In the same paper~\cite{AK78}, Ahlswede--Katona proved the following result for bipartite graphs. 
\begin{theorem}[Ahlswede--Katona~\cite{AK78}]\label{Thm:AK78-bipartite}
    Let $r, s, m$ be nonnegative integers with $m \le rs$ and $r \ge s$. 
    Suppose that $G$ is an $r \times s$ bipartite graph with $m$ edges. 
    Then 
    \begin{align*}
        N(S_2,G) \le N\big( S_2, B(r,s,m) \big). 
    \end{align*}
\end{theorem}

Let $r, s, \ell, k, m$ be integers with $r \ge \ell$, $s \ge k$, and $k \ell \le m \le r k$. 

Write $m = k \ell + p(r - \ell) + q$, where $p$ and $q$ are nonnegative integers with $q \in [0, r -\ell -1]$. 
Define $B_{1}(r, s, m, \ell, k)$ as the bipartite graph with parts $U = \{u_1, \ldots, u_{r} \}$ and $W = \{w_1, \ldots, w_{s} \}$, and with edge set:
\begin{enumerate}[label=(\roman*)]
    \item $\big\{ \{u_i, w_j\} \colon (i,j) \in [r] \times [p] \big\}$, 
    \item $\big\{ \{u_i, w_{p+1}\} \colon i \in [\ell + q] \big\}$, and 
    \item $\big\{ \{u_i, w_j\} \colon (i,j) \in [\ell] \times [p+2, k] \big\}$. 
\end{enumerate}

Write $m = k \ell + p k + q$, where $p$ and $q$ are nonnegative integers with $q \in [0, k - 1]$. 
Define $B_{2}(r, s, m, \ell, k)$ as the bipartite graph with parts $U = \{u_1, \ldots, u_{r} \}$ and $W = \{w_1, \ldots, w_{s} \}$, and with edge set
\begin{align*}
    \big\{ \{u_i, w_j\} \colon (i,j) \in [\ell+p] \times [k] \big\}
    \cup \big\{ \{u_{\ell+p+1}, w_i\} \colon i \in [q] \big\}. 
\end{align*}

\begin{theorem}\label{Thm:biparite-ell-k-Z_{1}(G)}
    Let $r, s, \ell, k, m$ be integers with $r \ge s$, $\ell \ge k$, $r \ge \ell$, $s \ge k$, and $k \ell \le m \le rs$. 
    Let $G = G[U, W]$ be an $r\times s$ bipartite graph with $m$ edges. 
    Suppose that there exists a subset $I \subseteq U$ with $|I| = \ell$ and $\delta_{G}(I) \ge k$. 
    Then 
    \begin{align*}
        N(S_2, G)
        \le 
        \begin{cases}
            N\big( S_2, B_1(r,s,m,\ell,k) \big), &\quad\text{if}\quad m \le rk ~\text{and}~ k + \ell \le r, \\[0.5em]
            N\big( S_2, B_2(r,s,m,\ell,k) \big), &\quad\text{if}\quad m \le rk ~\text{and}~ k + \ell > r, \\[0.5em]
            N\big( S_2, B(r,s,m) \big), &\quad\text{if}\quad m \ge r k. 
        \end{cases}
    \end{align*}
\end{theorem}

An argument analogous to that in the proof of Theorem~\ref{Thm:biparite-ell-k-Z_{1}(G)} yields the following result.
\begin{theorem}~\label{Thm:biparite-k-ell-Z_{1}(G)}
    Let $r, s, \ell, k, m$ be integers with $r \ge s$, $\ell \ge k$, $r \ge \ell$, $s \ge k$, and $k \ell \le m \le rs$.  
    Let $G = G[U, W]$ be an $r\times s$ bipartite graph with $m$ edges. 
    Suppose that there exists a subset $I \subseteq W$ with $|I| = k$ and $\delta_{G}(I) \ge \ell$. 
    Then 
    \begin{align*}
        N(S_2, G)
        \le 
        \begin{cases}
            N\big( S_2, B_1(r,s,m,\ell,k) \big), &\quad\text{if}\quad m \le rk ~\text{and}~ k + \ell \le r, \\[0.5em]
            N\big( S_2, B_2(r,s,m,\ell,k) \big), &\quad\text{if}\quad m \le rk ~\text{and}~ k + \ell > r, \\[0.5em]
            N\big( S_2, B(r,s,m) \big), &\quad\text{if}\quad m \ge r k. 
        \end{cases}
    \end{align*}
\end{theorem}

This paper is organized as follows. 
Section~\ref{SEC:Prelim} introduces some definitions and preliminary results. 
In Section~\ref{SEC:proof-bipartite}, we present the proof of Theorem~\ref{Thm:biparite-ell-k-Z_{1}(G)}. 
In Section~\ref{Sec:note-general-graph-alpha-alpha}, we present the proof of Theorem~\ref{Thm:Max-S2-alpha-alpha}. 
The proof of Theorem~\ref{Thm:Max-S2-alpha-beta} is provided in Section~\ref{Sec:note-general-graph-alpha-beta}. 

\section{Preliminaries}\label{SEC:Prelim}
Let $G$ be a graph. 
We use $N_{G}(v)$ to denote the set of neighbors of a vertex $v$ in $G$. 
The subscript $G$ will be omitted when it is clear from the context. 
Given two disjoint sets $U$ and $W$, we use $K[U,W]$ to denote the complete bipartite graph with parts $U$ and $W$. 
For a bipartite graph $G = G[U,W]$, let 
\begin{align*}
\Delta_{G, \mathrm{left}} 
\coloneqq \max\{d_{G}(v) \colon v \in U\}, 
\quad \text{and} \quad 
\Delta_{G, \mathrm{right}} 
\coloneqq \max\{d_{G}(v) \colon v \in W\}.
\end{align*}

Following the definition in~\cite{GT72}, we define the \emph{first Zagreb index} of a graph $G$ as  
\begin{align*}
    Z_{1}(G)
    \coloneqq \sum_{v \in V(G)} d_{G}^{2}(v). 
\end{align*}
Since $N(S_2, G) = \sum_{v \in V(G)} \tbinom{d_{G}(v)}{2}$, it follows that 
\begin{align}\label{equ:Zagreb-index-S2}
    Z_{1}(G) = 2 N(S_2, G) + 2|G|. 
\end{align}
Therefore, in all theorems from Section~\ref{Sec:Introduction}, one may replace $N(S_{2}, \cdot)$ with $Z_{1}(\cdot)$. 

\begin{fact}\label{FACT:Z_{1}(G)-del-uv-add-xy}
    Let $G$ be a graph. 
    Suppose that $u v \in G$ and $xy \not\in G$. 
    Consider the new graph $\hat{G} \coloneqq \big( G \setminus \{uv\} \big) \cup \{xy\}$. 
    Then the following hold: 
    \begin{enumerate}[label=(\roman*)]
        \item\label{FACT:Z_{1}(G)-del-uv-add-xy-a} If $\{u,v\} \cap \{x, y\} = \emptyset$, then 
        \begin{align*}
            Z_{1}(\hat{G}) - Z_{1}(G)
            = 2 \big( d_{G}(x) + d_{G}(y) - d_{G}(u) - d_{G}(v) \big) + 4. 
        \end{align*}
        \item\label{FACT:Z_{1}(G)-del-uv-add-xy-b} If $\{u,v\} \cap \{x, y\} \neq \emptyset$, then  
        \begin{align*}
            Z_{1}(\hat{G}) - Z_{1}(G)
            = 2 \big( d_{G}(x) + d_{G}(y) - d_{G}(u) - d_{G}(v) \big) + 2. 
        \end{align*}
    \end{enumerate}
\end{fact}

A simple calculation yields the following fact. 

\begin{fact}\label{FACT:Z_1(B_1)-and-Z_1(B_2)}
    We have 
    \begin{align*}
    Z_{1} \big(B_1(r,s,m,\ell,k)\big) 
    = \ell k^2 + 2m \ell - k \ell^2 + Z_{1}\big( B(r-\ell, k, m - k\ell) \big), 
    \end{align*}
    and
    \begin{align*}
    Z_{1} \big(B_2(r,s,m,\ell,k)\big) 
    = \ell k^2 + 2m \ell - k \ell^2 + Z_{1}\big( B(k, r-\ell, m - k\ell) \big).
    \end{align*}
\end{fact}

\section{Proof of Theorem~\ref{Thm:biparite-ell-k-Z_{1}(G)}}\label{SEC:proof-bipartite}
In this section, we present the proof of Theorem~\ref{Thm:biparite-ell-k-Z_{1}(G)}. 
The proof of Theorem~\ref{Thm:biparite-k-ell-Z_{1}(G)} is essentially identical to that of Theorem~\ref{Thm:biparite-ell-k-Z_{1}(G)}, and is therefore omitted. 

Let $\mathcal{B}(r,s,\ell, k,m)$ be the collection of all $r \times s$ bipartite graphs $G = G[U,W]$ with exactly $m$ edges such that there exists a subset $I \subseteq U$ with $|I| = \ell$ and $\delta_{G}(I) \ge k$. 
Define 
\begin{align*}
    \phi(r,s,\ell, k,m)
    & \coloneqq \max\big\{ Z_{1}(G) \colon G \in \mathcal{B}(r,s,\ell, k,m) \big\}  \quad\text{and}\quad \\[0.3em]
    \Phi(r,s,\ell, k,m)
    & \coloneqq \big\{ G \in \mathcal{B}(r,s,\ell, k,m) \colon Z_{1}(G) = \phi(r,s,\ell, k,m) \big\}. 
\end{align*}

\begin{proof}[Proof of Theorem \ref{Thm:biparite-ell-k-Z_{1}(G)}]
    Fix integers $r, s, \ell, k, m$ with $r \ge s$, $\ell \ge k$, $r \ge \ell$, $s \ge k$, and $k \ell \le m \le rs$.   
    Let $G = G[U, W]$ be an $r \times s$ bipartite graph with $m$ edges, and let $I \subseteq U$ be a subset with $|I| = \ell$ and $\delta_{G}(I) \ge k$. 
    By~\eqref{equ:Zagreb-index-S2}, it suffices to show that 
    \begin{align*}
        Z_{1}(G)
        \le 
        \begin{cases}
            Z_{1}\big(B_1(r,s,m,\ell,k) \big), &\quad\text{if}\quad m \le rk ~\text{and}~ k + \ell \le r, \\[0.5em]
            Z_{1}\big(B_2(r,s,m,\ell,k) \big), &\quad\text{if}\quad m \le rk ~\text{and}~ k + \ell > r, \\[0.5em]
            Z_{1}\big(B(r,s,m) \big), &\quad\text{if}\quad m \ge r k. 
        \end{cases}
    \end{align*}
    We may assume that $G \in \Phi(r,s,\ell, k,m)$, that is, $G$ is an extremal graph with respect to $Z_{1}(G)$. 
    If $m \ge r k$, then the theorem follows from~\eqref{equ:Zagreb-index-S2} and Theorem~\ref{Thm:AK78-bipartite}. 
    Thus, we may assume that $m \le rk$. 

    Set $U = \{u_1, \ldots, u_{r}\}$ and $W = \{w_1, \ldots, w_{s}\}$. 
    By relabeling the vertices, we may assume that 
    \begin{align}\label{equ:degree-sequence-bipartite-decending}
        d_{G}(u_1) \ge \cdots \ge d_{G}(u_{r})
        \quad\text{and}\quad 
        d_{G}(w_1) \ge \cdots \ge d_{G}(w_{s}). 
    \end{align}
    We may assume that $I = \{u_1, \ldots, u_{\ell}\}$. 
    It follows from the assumption on $I$ that 
    \begin{align}\label{equ:degree-sequence-I-bipartite}
        d_{G}(u_1) \ge \cdots \ge d_{G}(u_{\ell}) \ge k. 
    \end{align}

    The proof proceeds by induction on $k$. 
    The base case $k=1$ holds by~\eqref{equ:Zagreb-index-S2} and Theorem~\ref{Thm:AK78-bipartite}. 
    Now assume that $k \ge 2$. 
    The following claim is an easy consequence of Fact~\ref{FACT:Z_{1}(G)-del-uv-add-xy} and assumption~\eqref{equ:degree-sequence-bipartite-decending}. 

    \begin{claim}\label{Claim:complete-bipartite-graph-bi}
        Suppose that $u_iw_j \in G$ for some $(i,j) \in [r] \times [s]$. Then $u_{i}w_{j'} \in G$ for every $j' \in [j-1]$. 
        Consequently, 
        \begin{align}\label{equ:Nws-subseteq}
            N_{G}(w_1) \supseteq \cdots \supseteq N_{G}(w_{s}).  
        \end{align} 
    \end{claim}

    \begin{proof}[Proof of Claim~\ref{Claim:complete-bipartite-graph-bi}]
    Suppose to the contrary that $u_iw_j \in G$ for some $(i,j) \in [r] \times [s]$, but $u_{i}w_{j'} \not\in G$ for some $j' \in [j-1]$.  
        Let $\hat{G} \coloneqq \big(G \setminus \{u_iw_j\}\big) \cup \{u_{i}w_{j'}\}$. 
        It follows from Fact~\ref{FACT:Z_{1}(G)-del-uv-add-xy} and assumption~\eqref{equ:degree-sequence-bipartite-decending} that 
        \begin{align*}
            Z_{1}(\hat{G}) - Z_{1}(G)
            = 2\big(d_{G}(w_{j'}) - d_{G}(w_j)\big) + 2 
            > 0. 
        \end{align*}
        Therefore, to derive a contradiction, it suffices to show that $\hat{G} \in \mathcal{B}(r,s,\ell,k,m)$. 
        Note that $d_{G}(u_i) = d_{\hat{G}}(u_i)$ for every $u_i \in U$. 
        Therefore, by~\eqref{equ:degree-sequence-I-bipartite}, we have $d_{\hat{G}}(u_i) \ge k$ for every $i \in [\ell]$, and then $\hat{G} \in \mathcal{B}(r,s,\ell,k,m)$. 
        This completes the proof of Claim~\ref{Claim:complete-bipartite-graph-bi}. 
    \end{proof}

    It follows from Claim~\ref{Claim:complete-bipartite-graph-bi} that for every $u_i \in U$, 
    \begin{align}\label{equ:bipartite-neighbor-ui}
        N_{G}(u_i) = \big\{w_1, \ldots, w_{d_{G}(u_i)}\big\}. 
    \end{align}
    Combining it with assumption~\eqref{equ:degree-sequence-bipartite-decending}, we obtain 
    \begin{align}\label{equ:Nur-subseteq}
        N_{G}(u_1) \supseteq \cdots \supseteq N_{G}(u_{r}). 
    \end{align}
    Additionally, it follows from Claim~\ref{Claim:complete-bipartite-graph-bi} and~\eqref{equ:Nur-subseteq} that for every $w_j \in W$, 
    \begin{align}\label{equ:bipartite-neighbor-wj}
        N_{G}(w_j) = \big\{u_1, \ldots, u_{d_{G}(w_j)}\big\}. 
    \end{align}
    
    Let $A \coloneq U \setminus N_{G}(w_1)$ and $B \coloneq W \setminus N_{G}(u_1)$. 
    It follows from~\eqref{equ:Nws-subseteq} and~\eqref{equ:Nur-subseteq} that both $A$ and $B$ are isolated (a set $S \subseteq V(G)$ is \emph{isolated} if $d_{G}(v) = 0$ for every $v \in S$). 
    Let $J \coloneqq \{w_1, \ldots, w_{k}\}$. 
    It follows from~\eqref{equ:degree-sequence-I-bipartite} and~\eqref{equ:bipartite-neighbor-ui} that 
    \begin{align*}
        K[I, J] \subseteq G. 
    \end{align*}

    \begin{claim}\label{CLAIM:dw1>=du1}
        We may assume that $\Delta_{G, \mathrm{right}} \ge \Delta_{G, \mathrm{left}}$.  
    \end{claim}
    \begin{proof}[Proof of Claim~\ref{CLAIM:dw1>=du1}]
        Suppose that $\Delta_{G, \mathrm{right}} < \Delta_{G, \mathrm{left}}$, that is, $d_{G}(w_1) < d_{G}(u_1)$.  
        We will construct a new graph $G^{\ast} \in \Phi(r,s,\ell,k,m)$ with $\Delta_{G^{\ast}, \mathrm{right}} \ge \Delta_{G^{\ast}, \mathrm{left}}$. 

        First, suppose that $d_{G}(u_{k}) \ge \ell$. Then it follows from~\eqref{equ:bipartite-neighbor-ui} and~\eqref{equ:Nur-subseteq} that $d_{G}(w_j) \ge k$ for every $j \in [\ell]$. 
        Since $d_{G}(w_1) \le d_{G}(u_1) - 1 \le s - 1$, it follows from~\eqref{equ:Nws-subseteq} and~\eqref{equ:bipartite-neighbor-wj} that the set $\{u_s, \ldots, u_{r}\}$ is isolated. 
        Let $\hat{U} \coloneqq \{u_1, \ldots, u_{s}\}$ and $\hat{W} \coloneqq W \cup \{u_{s+1}, \ldots, u_{r}\}$. 
        Let $G^{\ast} = G^{\ast}[\hat{W}, \hat{U}]$. 
        Then it follows from the argument above that $G^{\ast} \in \Phi(r,s,\ell,k,m)$ and $\Delta_{G^{\ast}, \mathrm{right}} = \Delta_{G, \mathrm{left}} > \Delta_{G, \mathrm{right}} = \Delta_{G^{\ast}, \mathrm{left}}$. 

        Now suppose that $d_{G}(u_{k}) \le \ell - 1$.
        Then it follows from $d_{G}(w_{k}) \ge \ell$ that $d_{G}(w_{k}) > d_{G}(u_{k})$. 
        Let $t \in [k]$ be the smallest index such that $d_{G}(w_t) \ge d_{G}(u_t)$. 
        It follows from the assumption $d_{G}(w_1) < d_{G}(u_1)$ that $t \in [2,k]$. 
        Let $T \coloneqq d_{G}(w_t)$. 
        Then it follows from~\eqref{equ:degree-sequence-bipartite-decending} that $T = d_{G}(w_t) \ge d_{G}(w_{k}) \ge \ell$. 

        It follows from~\eqref{equ:Nws-subseteq} that 
        \begin{align*}
            K[\{w_1, \ldots, w_{t-1}\}, \{u_1,\ldots,u_{T}\}]
            \subseteq G. 
        \end{align*}

        It follows from the definition of $t$ that $d_{G}(u_{t-1}) > d_{G}(w_{t-1}) \ge d_{G}(w_{t}) = T$, and hence, 
        \begin{align*}
            K[\{u_1, \ldots, u_{t-1}\}, \{w_1, \ldots, w_{T}\}]
            \subseteq G. 
        \end{align*}

        It follows from the definition of $T$ that $\{w_{t}, u_{T+1}\} \not\in G$, and it follows from the fact $d_{G}(u_t) \le d_{G}(w_t)$ that $\{u_{t}, w_{T+1}\} \not\in G$. 

        We now construct $G^{\ast}$ in the following way 
        \begin{enumerate}[label=(\roman*)]
            \item For every $j \in [t-1]$, delete from $G$ the edge set 
            \begin{align*}
                \big\{u_i w_j \colon i \in [T+1, d_{G}(w_j)] \big\} \cup \big\{u_j w_i \colon i \in [T+1, d_{G}(u_j)] \big\}. 
            \end{align*}
            \item Then for every $j \in [t-1]$, add the following sets into the edge set 
            \begin{align*}
                \big\{u_i w_j \colon i \in [T+1, d_{G}(u_j)] \big\} \cup \big\{u_j w_i \colon i \in [T+1, d_{G}(w_j)] \big\}. 
            \end{align*}
        \end{enumerate}

        It is clear from the definition above that 
        \begin{enumerate}[label=(\roman*)]
            \item \label{ITEM:u_1-change-w_1} $\big( d_{G^{\ast}}(u_i), d_{G^{\ast}}(w_i) \big) = \big( d_{G}(w_i), d_{G}(u_i) \big)$ for $i \in [t-1]$, 
            \item $\big( d_{G^{\ast}}(u_i), d_{G^{\ast}}(w_i) \big) = \big( d_{G}(u_i), d_{G}(w_i) \big)$ for $i \in [t, T]$, 
            \item $\big( d_{G^{\ast}}(u_i), d_{G^{\ast}}(w_i) \big) = \big( d_{G}(w_i), d_{G}(u_i) \big)$ for $i \in [T+1, s]$. 
            \item $d_{G^{\ast}}(u_i) = d_{G}(u_i) = 0$ for $ i > s$. 
        \end{enumerate}
        Consequently, we have $|G^{\ast}| =|G|$, and $Z_{1}(G^{\ast})=Z_{1}(G)$. 
        For the set $I = \{u_1, \ldots, u_{\ell}\}$, since $t \le k \le \ell \le T$, we have 
        \begin{align*}
        \delta_{G^{\ast}}(I) 
        \ge \min\{d_{G}(w_{t-1}), d_{G}(u_{\ell})\} 
        \ge \min\{T, d_{G}(u_{\ell})\}
        \ge k. 
        \end{align*}
        Therefore, $G^{\ast} \in \Phi(r,s,\ell,k,m)$. 
        Besides, it follows~\ref{ITEM:u_1-change-w_1} that $\Delta_{G^{\ast}, \mathrm{right}} = \Delta_{G, \mathrm{left}} > \Delta_{G, \mathrm{right}} = \Delta_{G^{\ast}, \mathrm{left}}$. 
        This completes the proof of Claim~\ref{CLAIM:dw1>=du1}. 
\end{proof}

    By Claim~\ref{CLAIM:dw1>=du1}, we may additionally assume that $G$ satisfies $\Delta_{G, \mathrm{right}} \ge \Delta_{G, \mathrm{left}}$, and $\Delta_{G, \mathrm{right}}$ is maximized among all such graphs in $\Phi(r,s,\ell,k,m)$. 

    Recall that $A \coloneq U \setminus N_{G}(w_1)$ and $B \coloneq W \setminus N_{G}(u_1)$. 
    Let $U_1 \coloneqq U \setminus A$ and $W_1 \coloneqq W \setminus \big(B \cup \{w_1\}\big)$. 
    Let $G_1 \coloneqq G[U_1, W_1]$ denote the induced subgraph of $G$ on $U_1 \cup W_1$. 
    Since $A \cup B$ is isolated in $G$, we have $|G_1| = |G| - d_{G}(w_1) = m - d_{G}(w_1)$. 
    Moreover, 
    \begin{align}\label{equ:ZG-and-ZG1}
        Z_{1}(G) 
        & = \sum_{i=1}^{r}d_{G}(u_i)^2 + \sum_{j=1}^{s} d_{G}(w_j)^2 \notag \\ 
        & =\sum_{i=1}^{d_{G}(w_1)}d_{G}(u_i)^2 + \sum_{j=1}^{d_{G}(u_1)} d_{G}(w_j)^2 \notag \\
        & =\sum_{i=1}^{d_{G}(w_1)} \big( (d_{G}(u_i)-1)^2  + 2d_{G}(u_i) - 1 \big)  + \sum_{j=2}^{d_{G}(u_1)} d_{G}(w_j)^2 +  d_{G}(w_1)^2 \notag \\
        &= Z_{1}(G_1)+  2m  +d(w_1)^2 - d(w_1).
    \end{align}

    Let 
    \begin{align*}
        \big(\hat{r}, \hat{s}, \hat{\ell}, \hat{k}, \hat{m}\big)
        \coloneqq \big(d_{G}(w_1), d_{G}(u_1) - 1, \ell, k-1, m-d_{G}(w_1)  \big). 
    \end{align*}
    Note that $G_1 \in \mathcal{B}\big(\hat{r}, \hat{s}, \hat{\ell}, \hat{k}, \hat{m}\big)$, and it follows from Claim~\ref{CLAIM:dw1>=du1} that $\hat{r} = \Delta_{G, \mathrm{right}} \ge \Delta_{G, \mathrm{left}} = \hat{s} + 1$. 
    Since $\hat{k} = k-1$, it follows from the inductive hypothesis that
    \begin{align}\label{equ:induction-ZG1}
        Z_{1}(G_1) \le 
        \begin{cases}
            Z_{1}\big( B_1(\hat{r}, \hat{s}, \hat{m}, \hat{\ell}, \hat{k}) \big), &\quad\text{if}\quad \hat{m} \le \hat{r} \hat{k} ~\text{and}~ \hat{k} + \hat{\ell} \le \hat{r}, \\[0.5em]
            Z_{1}\big( B_2(\hat{r}, \hat{s}, \hat{m}, \hat{\ell}, \hat{k}) \big), &\quad\text{if}\quad \hat{m} \le \hat{r} \hat{k} ~\text{and}~ \hat{k} + \hat{\ell} > \hat{r}, \\[0.5em]
            Z_{1}\big( B(\hat{r}, \hat{s}, \hat{m}) \big), &\quad\text{if}\quad \hat{m} \ge \hat{r} \hat{k}. 
        \end{cases}
    \end{align}

    \begin{claim}\label{THM-bi-degree-du1}
        We may assume that $\Delta_{G, \mathrm{left}} = k$. 
    \end{claim}
    
    \begin{proof}[Proof of Claim~\ref{THM-bi-degree-du1}]
        Suppose that $\hat{m} \le \hat{r} \hat{k}$. 
        Then let $G^{\ast}$ denote the graph obtained from $G$ by replacing $G_1$ with $B_1(\hat{r}, \hat{s}, \hat{\ell}, \hat{k}, \hat{m})$ or $B_2(\hat{r}, \hat{s}, \hat{\ell}, \hat{k}, \hat{m})$ (depending on whether $\hat{k} + \hat{\ell} \le \hat{r}$).
        It follows from the definition of $B_1(\cdot)$ and $B_{2}(\cdot)$ that $G^{\ast} \in \mathcal{B}(r, s, \ell, k, m)$. 
        Moreover, it follows from~\eqref{equ:ZG-and-ZG1} and~\eqref{equ:induction-ZG1} that $Z_{1}(G^{\ast}) \ge Z_{1}(G)$. 
        Therefore, $G^{\ast} \in \Phi(r,s,\ell,k,m)$. 
        Note that $\Delta_{G^{\ast}, \mathrm{left}} = \hat{k} +1 =k$. 
        We are done. 

        Now suppose that $\hat{r} \hat{k} < \hat{m} = m-d_{G}(w_1)$. 
        Then it follows that 
        \begin{align*}
            d_{G}(w_1)
            = \hat{r}
            < \frac{m}{\hat{k}+1}
            \le \frac{rk}{k}
            = r, 
        \end{align*}
        which implies that $d_{G}(w_1) \le r-1$. 
        
        Let $G^{\ast}$ be the graph obtained from $G$ by replacing $G_1$ with $B(\hat{r}, \hat{s}, \hat{m})$. 
        Then $|G^{\ast}| = |G|$. 
        Since $\hat{m} > \hat{r}\hat{k}$, the definition of $B(\hat{r}, \hat{s}, \hat{m})$ implies that there are at least $\hat{k}$ vertices in $W_1$ with degree $\hat{r} = d_{G}(w_1) \ge \ell$. 
        Together with the vertex $w_1$, there are at least $\hat{k} + 1 = k$ vertices in $W$ adjacent to every vertex in $I$. 
        Consequently, $\delta_{G^{\ast}}(I) \ge k$, which implies $G^{\ast} \in \mathcal{B}(r,s,\ell,k,m)$. 
        Moreover, it follows from~\eqref{equ:ZG-and-ZG1} and~\eqref{equ:induction-ZG1} that $Z_{1}(G^{\ast}) \ge Z_{1}(G)$. 
        Therefore, $G^{\ast} \in \Phi(r,s,\ell,k,m)$. 
        
        Let $p$ and $q$ be nonnegative integers such that $\hat{m} = p \hat{r} + q$ and $q \in [1, \hat{r}]$. 
        Since $\hat{m} > \hat{r} \hat{k}$, we have $p \ge \hat{k}$. 
        We now consider the two cases separately: $q \ge p+1$ and $q \le p$.         

        Suppose that $q \ge p+1$. 
        Then let $\hat{G}$ be the graph obtained from $G^{\ast}$ by removing the edge set $\{u_i w_{p+2} \colon i \in [q-p, q]\}$ and then adding the set $\{u_{\hat{r}+1} w_i \colon i \in [p+1]\}$. 
        It is clear that $|\hat{G}| = |G^{\ast}| = |G|$, and 
        \begin{align*}
            Z_{1}(\hat{G}) - Z_{1}(G^{\ast})
            & = \sum_{i \in [p+2]} \big( d_{\hat{G}}^{2}(w_i) - d_{G}^{2}(w_i) \big) \\
            & \qquad + \sum_{j \in [q-p, q]} \big( d_{\hat{G}}^{2}(u_j) - d_{G}^{2}(u_j) \big) + d_{\hat{G}}^{2}(u_{\hat{r}+1}) - d_{G}^{2}(u_{\hat{r}+1}) \\
            & = (p+1)\big( (\hat{r}+1)^2 - \hat{r}^2 \big) \\
            & \qquad + (q-p-1)^2 - q^2 + (p+1) \big( (p+1)^2 - (p+2)^2 \big) + (p+1)^2 \\
            & = 2(p+1)(\hat{r} - q)
            \ge 0. 
        \end{align*}
        For the set $I = \{u_1, \ldots,u_{\ell}\}$, we have $\delta_{\hat{G}}(I) \ge p+1 \ge \hat{k}+1 =k$. 
        Hence, $\hat{G} \in \mathcal{B}(r,s,\ell,k,m)$. 
        Moreover, it follows from $Z_{1}(\hat{G}) \ge Z_{1}(G^{\ast})$ that $\hat{G} \in \Phi(r,s,\ell,k,m)$. 
        Note that $\hat{G}$ also satisfies $\Delta_{\hat{G}, \mathrm{right}} \ge \Delta_{\hat{G}, \mathrm{left}}$. 
        However, $\Delta_{\hat{G},\mathrm{right}} = d_{\hat{G}}(w_1) = d_{G}(w_1) + 1 > \Delta_{G,\mathrm{right}}$, contradicting the maximality of $\Delta_{G,\mathrm{right}}$. 

        Now suppose that $q \le p$. 
        Similar to the argument above, let $\hat{G}$ be the graph obtained from $G^{\ast}$ by removing the edge set $\{u_i w_{p+2} \colon i \in [q]\}$ and then adding the set $\{u_{\hat{r}+1} w_j \colon j \in [q]\}$. 
        It is clear that $|\hat{G}| = |G^{\ast}| = |G|$, and 
        \begin{align*}
            Z_{1}(\hat{G}) - Z_{1}(G^{\ast})
            & = \sum_{i \in [q]} \big( d_{\hat{G}}^{2}(w_i) - d_{G}^{2}(w_i) \big) + d_{\hat{G}}^{2}(w_{p+2}) - d_{G}^{2}(w_{p+2}) \\
            & \qquad + \sum_{j \in [q-p, q]} \big( d_{\hat{G}}^{2}(u_j) - d_{G}^{2}(u_j) \big) + d_{\hat{G}}^{2}(u_{\hat{r}+1}) - d_{G}^{2}(u_{\hat{r}+1}) \\
            & = q \big( (\hat{r}+1)^2 - \hat{r}^2 \big) - q^2 + q \big( (p+1)^2 - (p+2)^2 \big) + q^2 \\
            & = 2q(\hat{r} - p-1) 
            \ge 0, 
        \end{align*}
        where the last inequality follows from Claim~\ref{CLAIM:dw1>=du1} that $p+1 \le d_{G}(u_1) \le d_{G}(w_1) = \hat{r}$. 
        For the set $I = \{u_1, \ldots,u_{\ell}\}$, we have $\delta_{\hat{G}}(I) \ge p+1 \ge \hat{k}+1 =k$. 
        Hence, $\hat{G} \in \mathcal{B}(r,s,\ell,k,m)$. 
        Moreover, it follows from $Z_{1}(\hat{G}) \ge Z_{1}(G^{\ast})$ that $\hat{G} \in \Phi(r,s,\ell,k,m)$. 
        Notice that $\hat{G}$ also satisfies $\Delta_{\hat{G}, \mathrm{right}} \ge \Delta_{\hat{G}, \mathrm{left}}$. 
        However, since $q \ge 1$, we have $\Delta_{\hat{G},\mathrm{right}} = \hat{r} + 1 > \hat{r} = \Delta_{G,\mathrm{right}}$, contradicting the maximality of $\Delta_{G,\mathrm{right}}$.  
    \end{proof}

    Fix $G$ such that additionally, $d_{G}(u_1) = \Delta_{G, \mathrm{left}} = k$. 
    Recall that the subgraph induced by $I\cup \{w_1,\ldots,w_{k}\}$ is an $\ell$ by $k$ complete bipartite graph. 
    Let $G_2$ be the subgraph induced by $\{u_{\ell +1}, \ldots, u_r\}\cup \{w_1,\ldots,w_{k}\}$. 
    It follows from~\eqref{equ:Zagreb-index-S2} and Theorem~\ref{Thm:AK78-bipartite} that 
    \begin{align*}
        Z_{1}(G_2) \le 
        \begin{cases}
            Z_{1}\big( B(r-\ell, k, m - k\ell) \big), & \quad\text{if}\quad r-\ell  \ge k, \\[0.5em]
            Z_{1}\big( B(k,r-\ell,m-k\ell) \big), & \quad\text{if}\quad r-\ell  \le k. 
        \end{cases}
    \end{align*}

    Since $d_{G}(u_1) =k$, we have $d_{G}(w_i) = 0$ for $i \ge k+1$. 
    It follows that 
    \begin{align*}
        Z_{1}(G) 
        &= \sum_{i=1}^{r} d_{G}(u_i)^2 + \sum_{j=1}^{k} d_{G}(w_j)^2 \\
        &= \sum_{i=1}^{\ell } d(u_i)^2 + \sum_{i=\ell +1}^{r} d(u_i)^2 + \sum_{j=1}^{k}(\ell +d(w_j)-\ell )^2 \\
        & = \ell  k^2+ 2m\ell - k \ell^2  + \sum_{i=\ell +1}^{r} d(u_i)^2 + \sum_{j=1}^{k} \big( d(w_j)-\ell \big)^2 \\
        & = \ell  k^2+ 2m\ell - k \ell^2 + Z_{1}(G_2),  
    \end{align*}
    which, by Fact~\ref{FACT:Z_1(B_1)-and-Z_1(B_2)}, proves Theorem~\ref{Thm:biparite-ell-k-Z_{1}(G)}. 
\end{proof}

\section{Proof of Theorem~\ref{Thm:Max-S2-alpha-alpha}}\label{Sec:note-general-graph-alpha-alpha}

In this section,we present the proof of Theorem~\ref{Thm:Max-S2-alpha-alpha}. 
Since we are concerned only with the asymptotic $S_2$-density as $n \to \infty$, we omit floor notation in the proof for simplicity. 

Recall that $\mathcal{G}\big(n,\rho \tbinom{n}{2}, \alpha n, \beta n \big)$ denote the collection of all $n$-vertex graphs $G$ with exactly $\rho \tbinom{n}{2}$ edges such that there exists an independent set $I$ of size $\alpha n$ in $G$ such that $\delta_{G}(I) \ge \beta n$. 
For convenience, we define 
\begin{align*}
    \phi(n, \rho, \alpha, \beta)
    & \coloneqq \max\big\{ Z_{1}(G) \colon G \in \mathcal{G}\big(n,\rho \tbinom{n}{2}, \alpha n, \beta n \big) \big\}  \quad\text{and}\quad \\[0.3em]
    \Phi(n, \rho, \alpha, \beta)
    & \coloneqq \big\{ G \in \mathcal{G}\big(n,\rho \tbinom{n}{2}, \alpha n, \beta n \big) \colon Z_{1}(G) = \phi(n, \rho, \alpha, \beta) \big\}. 
\end{align*}

\begin{proof}[Proof of Theorem~\ref{Thm:Max-S2-alpha-alpha}]
Let $n$ be sufficiently large, $\rho \in \left[\frac{17}{25}, \frac{7}{10} \right]$ and $\alpha \in \left[\frac{17}{100}, \frac{23}{100}\right]$. 
Let $G$ be an $n$-vertex graph with $\rho \tbinom{n}{2}$ edges, containing an independent set $I \subseteq V(G)$ with $|I| = \alpha n$ and $\delta_{G}(I) \ge \alpha n$. 
By the lower bound given in Fact~\ref{FACT:lower-bound} and~\eqref{equ:Zagreb-index-S2}, it suffices to show that  
\begin{align*}
Z_{1}(G) \le 
\left(\alpha^3+(\rho-\alpha^2) \sqrt{\rho+\alpha^2}\right) n^3 + o(n^3).  
\end{align*}

Assume that $G \in \Phi(n, \rho, \alpha, \alpha)$, i.e., $G$ is an extremal graph. 
Let $I = \{u_1,\ldots,u_{\alpha n}\}$ be the independent set of $G$, and let $V \coloneqq V(G) \setminus I=\{v_1,\ldots, v_{(1-\alpha)n}\}$. 
By relabeling the vertices, we may assume that 
\begin{align}\label{EQ:alpha-alpha-degree-sequence}
d(u_1) \ge \cdots \ge d(u_{\alpha n}), 
\quad \text{and} \quad
d(v_1) \ge \cdots \ge d(v_{(1-\alpha)n}). 
\end{align}
It follows from $\delta_{G}(I) \ge \alpha n$ that
\begin{align}\label{EQ:alpha-alpha-degree>alpha-n}
d(u_1) \ge \cdots \ge d(u_{\alpha n}) \ge \alpha n. 
\end{align}

The following claim follows from Fact~\ref{FACT:Z_{1}(G)-del-uv-add-xy} and assumption~\eqref{EQ:alpha-alpha-degree-sequence}, and its proof is analogous to that of Claim~\ref{Claim:complete-bipartite-graph-bi}, we therefore omit the details and leave it to the interested reader. 

\begin{claim}\label{Claim:assume-G-after-shifting}
The following statements hold.  
\begin{enumerate}[label=(\roman*)]
\item \label{ITEM:alpha-alpha-ui-vj-complete} If $u_iv_j \in G$, then  $u_iv_{j'} \in G$ for every $j' \in [j-1]$. 
\item \label{ITEM:alpha-alpha-vi-vj-complete} If $v_iv_j\in G$, then $v_{i'}v_{j'} \in G$ for all $i' \le i$, $j' \le j$, and $i' \neq j'$.
\end{enumerate}
\end{claim}

It follows from Claim~\ref{Claim:assume-G-after-shifting}~\ref{ITEM:alpha-alpha-ui-vj-complete} that for every $u_i \in I$, 
\begin{align}\label{EQ:alpha-alpha-neigh-u_i}
N(u_i) = \big\{v_1, \ldots, v_{d(u_i)}\big\}. 
\end{align}
Combining it with assumption~\eqref{EQ:alpha-alpha-degree-sequence}, we obtain
\begin{align}\label{EQ:Nu1-contain-Nu2-alpha}
N(u_1) \supseteq \cdots \supseteq N(u_{\alpha n}). 
\end{align}

Additionally, it follows from Claim~\ref{Claim:assume-G-after-shifting}~\ref{ITEM:alpha-alpha-vi-vj-complete} that if $v_iv_j\in G$, then $v_{i'} v_j \in G$ for all $i' \le i$ with $i' \neq j$. 
It follows from~\eqref{EQ:alpha-alpha-degree>alpha-n} and~\eqref{EQ:alpha-alpha-neigh-u_i} that 
\begin{align*}
K[\{v_1, \ldots,v_{\alpha n}\}, \{u_1, \ldots,u_{\alpha n}\}] \subseteq G. 
\end{align*}

Denote the largest index $i \in [(1-\alpha)n]$ such that $v_{i-1}v_i\in  G$ by $\omega \coloneqq \omega(G)$. 
By the maximality of $\omega$ and Claim~\ref{Claim:assume-G-after-shifting}~\ref{ITEM:alpha-alpha-vi-vj-complete}, the subgraph induced by $V_1 \coloneqq  \{v_1,\ldots, v_{\omega}\}$ is complete, and the subgraph induced by $V_2\coloneqq  \{v_{\omega+1},\ldots, v_{(1-\alpha)n}\}$ is empty. 

Now we give the lower bound of $\omega$. 
Since $I$ is an independent set and $G[V_2]$ is empty, we have 
\begin{align*}
|G| 
= \rho \binom{n}{2} 
\le \binom{n}{2} - \binom{\alpha n}{2} -\binom{(1-\alpha)n -\omega}{2}
< \frac{1}{2}n^2 - \frac{1}{2}\alpha^2 n^2 -\frac{1}{2}\left((1-\alpha)n -\omega\right)^2, 
\end{align*}
equivalently, 
\begin{align*}
\omega 
> (1-\alpha)n - \left(n^2-\rho n^2 -\alpha^2 n^2 + \rho n\right)^{1/2}
= (1-\alpha)n - \left(1-\rho -\alpha^2 \right)^{1/2}n - o(n). 
\end{align*}
Recall that $\rho \in \left[\frac{17}{25}, \frac{7}{10} \right]$ and $\alpha \in \left[\frac{17}{100}, \frac{23}{100}\right]$.  
By some calculations, we have 
\begin{align*}
1-2 \alpha- \left(1-\rho -\alpha^2 \right)^{1/2} 
\ge 1-2\times \frac{23}{100}- \left(1-\frac{17}{25}-\left(\frac{23}{100}\right)^2 \right)^{1/2}
= \frac{54-\sqrt{2671}}{100} 
> \frac{1}{50}. 
\end{align*}
Thus, 
\begin{align}\label{EQ:lower-bound-omega}
\omega
\ge \alpha n + (1-2\alpha)n - \left(1-\rho -\alpha^2 \right)^{1/2}n - o(n)
\ge \alpha n + \frac{1}{50}n - o(n)
> \alpha n + \frac{1}{100}n. 
\end{align}

\begin{claim}\label{CLAIM:alpha-B_0-empty}
There exits a graph $G \in \Phi(n, \rho, \alpha, \alpha)$ such that $G[I \cup V_2]$ is empty. 
\end{claim}

\begin{proof}[Proof of Claim~\ref{CLAIM:alpha-B_0-empty}]
Let $B_0\coloneqq  G[I\cup V_2]$. 
Since $I$ is an independent set and $G[V_2]$ is empty, $B_0$ is bipartite graph. 
We now show that there exits a graph $G \in \Phi(n, \rho, \alpha, \alpha)$ such that $|B_0| = 0$. 
Assume, for contradiction, that no such graph $G$ exists. 
We choose $G\in \Phi(n, \rho, \alpha, \alpha)$ such that $|B_0|$ is as small as possible. Then $|B_0|\ge 1$. 
By Claim~\ref{Claim:assume-G-after-shifting} and~\eqref{EQ:Nu1-contain-Nu2-alpha}, we have $u_1v_{\omega+1}\in G$. 
Let 
\begin{align*}
i_1 \coloneqq \max \left\{i \colon u_1v_i \in G\right\}, 
\quad \text{and} \quad
j_1\coloneqq \min\left\{j \colon v_{\omega +1}v_{j} \not\in G \right\}. 
\end{align*} 
Then $i_1 = d(u_1) \ge \omega + 1$ and $j_1 \le \omega$ by the choice of $\omega$. 
Moreover, $v_{j_1}v_j \not\in G$ for every $j \ge \omega+1$ by Claim \ref{Claim:assume-G-after-shifting}~\ref{ITEM:alpha-alpha-vi-vj-complete}. 
Let $G^{\ast}$ be the graph obtained from $G$ by removing the edge set $\{u_1 v_i \colon i \in [\omega+1,i_1]\}$ and then adding the set $\{v_{j_1} v_i \colon i \in [\omega+1,i_1]\}$. 
It is clear that $|G^{\ast}| = |G|$, $\delta_{G^{\ast}}(I)\ge \alpha n$ by \eqref{EQ:lower-bound-omega}, and
\begin{align*}
Z_{1}(G^{\ast}) - Z_{1}(G)
& = \left(d(u_1)-(i_1-\omega)\right)^2 + \left(d(v_{j_1})+(i_1-\omega)\right)^2 - d(u_1)^2 - d(v_{j_1})^2 \\
& = 2(i_1 - \omega)(d(v_{j_1}) - \omega) 
\ge 0. 
\end{align*}
Therefore, we have $G^{\ast} \in \Phi(n, \rho, \alpha, \alpha)$. 
This contradicts the minimality of $|B_0|$. 
\end{proof}

By Claim~\ref{CLAIM:alpha-B_0-empty}, we may restrict our consideration to the graph $G \in \Phi(n, \rho, \alpha, \alpha)$ with $\left|G[I \cup V_2]\right| = 0$. 
We denote by $\omega_1$ and $\omega_2$ the minimum and maximum values of $\omega(G)$ among all such graphs, respectively. 
We distinguish the following two cases. 

\textbf{Case 1. $\omega_1 < n/2$. } 

Let 
\begin{align*}
J_{G} \coloneqq \left\{u_i \in I \colon u_i v_{\omega} \in G\right\}. 
\end{align*}
Now we choose $G \in \Phi(n, \rho, \alpha, \alpha)$  with $\omega(G) = \omega_1$ and minimizing $|J_G|$. 
First, we consider the case that $J(G) = \emptyset$. 
Let $G_1 = G[V_1 \setminus\{v_{\omega_1}\}]$, $G_2 = G[I \cup V_2\cup \{v_{\omega_1}\}]$, and $B_1$ be the bipartite graph induced by edges between $V_1 \setminus \{v_{\omega_1}\}$ and $I \cup V_2 \cup \{v_{\omega_1}\}$. 
Then $G_1$ is complete. 
By the definition of $\omega(G)$, $\left|G[I \cup V_2]\right| = 0$ and $J(G) = \emptyset$, $G_2$ is empty. 
It follows that
\begin{align}\label{EQ:Z_{1}(G)_equals_Z_{1}(B_1)}
Z_{1}(G) 
&= \sum_{i=1}^{\omega_1 -1} d(v_i)^2 + \sum_{i=\omega_1}^{(1-\alpha)n} d(v_i)^2 + \sum_{i=1}^{\alpha n} d(u_i)^2 \notag \\
&= \sum_{i=1}^{\omega_1 -1}\left(\omega_1 -2 + d_{B_1}(v_i)\right)^2 + \sum_{i=\omega_1}^{(1-\alpha)n} d_{B_1}(v_i)^2 + \sum_{i=1}^{\alpha n } d_{B_1}(u_i)^2 \notag\\
&= (\omega_1-1)(\omega_1 -2)^2 + 2(\omega_1 -2) \left(|G|- \binom{\omega_1-1}{2} \right) + Z_{1}(B_1). 
\end{align}

This means that for fixed $\omega_1$, it suffices to maximize $Z_{1}(B_1)$. 
Let 
\begin{align*}
\left(r,s,\ell,k,m\right) 
\coloneqq \left(n- \omega_1 +1, \omega_1 -1, \alpha n, \alpha n, |G|- \tbinom{\omega_1-1}{2}\right). 
\end{align*}
Since $\delta_{G}(I) \ge \alpha n$ and $r \ge s$ by $\omega_1 < n/2$, it follows from Theorem~\ref{Thm:biparite-ell-k-Z_{1}(G)} that 
\begin{align}\label{EQ:alpha-Z_1-B_1-from_bi}
 Z_{1}(B_1)
        \le 
        \begin{cases}
            Z_{1}\big(B_1(r,s,m,\ell,k) \big), &\quad\text{if}\quad m \le rk ~\text{and}~ k + \ell \le r, \\[0.5em]
            Z_{1}\big(B_2(r,s,m,\ell,k) \big), &\quad\text{if}\quad m \le rk ~\text{and}~ k + \ell > r, \\[0.5em]
            Z_{1}\big(B(r,s,m) \big), &\quad\text{if}\quad m \ge r k. 
        \end{cases}
\end{align}
Then let $G^{\ast}$ denote the graph obtained from $G$ by replacing $B_1$ with $B_1(r,s,m,\ell,k)$, $B_2(r,s,m,\ell,k)$ or $B(r,s,m)$ (depending on whether $m \le rk$ and $k + \ell \le r$). 
It follows from the definition of $B_1(\cdot)$, $B_2(\cdot)$ and $B(\cdot)$ that $G^{\ast} \in \mathcal{G}\big(n,\rho \tbinom{n}{2}, \alpha n, \alpha n \big)$. 
Moreover, it follows from~\eqref{EQ:Z_{1}(G)_equals_Z_{1}(B_1)} and~\eqref{EQ:alpha-Z_1-B_1-from_bi} that $Z_{1}(G^{\ast}) \ge Z_{1}(G)$. 
Therefore, $G^{\ast} \in \Phi(n, \rho, \alpha, \alpha)$. 
Note that $\left|G^{\ast}[I \cup V_2]\right| = \left|G[I \cup V_2]\right| = 0$. 
Hence, we have $|B_1| > (\omega_1 -2 )(n-\omega_1 +1)$. 
Otherwise, $v_{\omega_1-1} v_{\omega_1} \not\in G^{\ast}$, contradicting the minimality of $\omega_1$. 
Combining this with the upper bound $|B_1| \le (\omega_1 -1)(n- \omega_1 +1)$, we have  
\begin{align*}
\rho \binom{n}{2} 
= |G| 
= \binom{\omega_1 -1}{2} + |B_1|
= \frac{1}{2} \omega_1^2 + \omega_1 (n -\omega_1) + o(n^2), 
\end{align*}
which implies that 
\begin{align}\label{EQ:alpha-calculation-w_1}
\omega_1 
= \left(1-\sqrt{1-\rho} \right)n +o(n). 
\end{align}

It follows from~\eqref{EQ:Z_{1}(G)_equals_Z_{1}(B_1)} and~\eqref{EQ:alpha-Z_1-B_1-from_bi} that 
\begin{align}\label{EQ:alpha-calculation-w_1-Z_1(G)}
Z_{1}(G) 
\le \omega_1^3 + 2\omega_1 |B_1| + Z_{1}(B_1)
\le \omega_1^3 + 3\omega_1^2(n-\omega_1) + \omega_1(n-\omega_1)^2 + o(n^3). 
\end{align}

Combining~\eqref{EQ:alpha-calculation-w_1} and~\eqref{EQ:alpha-calculation-w_1-Z_1(G)}, we have  
\begin{align*}
Z_{1}(G) \le \left(2\rho -1 +(1-\rho)^{3/2}\right)n^3 + o(n^3). 
\end{align*}

Next, we consider the case $|J(G)| \ge 1$. 
Let $G_1'= G[V_1]$, $G_2'=G[I\cup V_2]$, and  $B_2$ be the bipartite graph induced by the edges between $V_1$ and $I \cup V_2$. 
Then $G_1'$ is complete and $G_2'$ is empty. 
It is similar to~\eqref{EQ:Z_{1}(G)_equals_Z_{1}(B_1)} that 
\begin{align}\label{EQ:Z_{1}(G)-B2-omega1}
Z_{1}(G) = \omega_1 (\omega_1 -1)^2 + 2(\omega_1 -1) \left(|G| - \binom{\omega_1}{2}\right) + Z_{1}(B_2). 
\end{align}
Let 
\begin{align*}
\left(\hat{r},\hat{s},\hat{\ell},\hat{k},\hat{m}\right) 
\coloneqq \left(n- \omega_1, \omega_1, \alpha n, \alpha n, |G|- \tbinom{\omega_1}{2}\right). 
\end{align*}
Similarly, it follows from Theorem~\ref{Thm:biparite-ell-k-Z_{1}(G)} that 
\begin{align}\label{EQ:alpha-Z_1-B_2-from_bi}
 Z_{1}(B_2)
        \le 
        \begin{cases}
            Z_{1}\big(B_1(\hat{r},\hat{s},\hat{m},\hat{\ell},\hat{k}) \big), &\quad\text{if}\quad \hat{m} \le \hat{r}\hat{k} ~\text{and}~ \hat{k} + \hat{\ell} \le \hat{r}, \\[0.5em]
            Z_{1}\big(B_2(\hat{r},\hat{s},\hat{m},\hat{\ell},\hat{k}) \big), &\quad\text{if}\quad \hat{m} \le \hat{r}\hat{k} ~\text{and}~ \hat{k} + \hat{\ell} > \hat{r}, \\[0.5em]
            Z_{1}\big(B(\hat{r},\hat{s},\hat{m}) \big), &\quad\text{if}\quad \hat{m} \ge \hat{r} \hat{k}. 
        \end{cases}
\end{align}

Then let $\hat{G}$ denote the graph obtained from $G$ by replacing $B_2$ with $B_1(\hat{r},\hat{s},\hat{m},\hat{\ell},\hat{k}$, $B_2(\hat{r},\hat{s},\hat{m},\hat{\ell},\hat{k})$ or $B(\hat{r},\hat{s},\hat{m})$ (depending on whether $\hat{m} \le \hat{r}\hat{k}$ and $\hat{k} + \hat{\ell} \le \hat{r}$). 
It follows from the definition of $B_1(\cdot)$, $B_2(\cdot)$ and $B(\cdot)$ that $\hat{G} \in \mathcal{G}\big(n,\rho \tbinom{n}{2}, \alpha n, \alpha n \big)$. 
Moreover, it follows from~\eqref{EQ:Z_{1}(G)-B2-omega1} and~\eqref{EQ:alpha-Z_1-B_2-from_bi} that $Z_{1}(\hat{G}) \ge Z_{1}(G)$. 
Therefore, $\hat{G} \in \Phi(n, \rho, \alpha, \alpha)$. 
Note that $\left|\hat{G}[I \cup V_2]\right| = \left|G[I \cup V_2]\right| = 0$, and $\omega(\hat{G}) = \omega(G) = \omega_1$. 
Then the subgraph induced by $V_1 \setminus \{v_{\omega}\}$ and $V_2 \cup I$ is a complete bipartite graph in $\hat{G}$. 
Otherwise, by~\eqref{EQ:lower-bound-omega}, we have $|J_{\hat{G}}| < |J_{G}|$, a contradiction to the minimality of $|J_{G}|$. 
It follows that $|B_2| \ge (\omega_1 -1)(n - \omega_1)$. 
Similar to the calculations of~\eqref{EQ:alpha-calculation-w_1} and~\eqref{EQ:alpha-calculation-w_1-Z_1(G)} above, we have 
\begin{align*}
Z_{1}(G) \le \left(2\rho -1 +(1-\rho)^{3/2}\right)n^3 + o(n^3). 
\end{align*}
In this case, by combining the calculations in Lemma~\ref{LEM:alpha-inequality-star-two-upper-bounds}, we obtain 
\begin{align*}
Z_{1}(G) 
\le \left(2\rho -1 +(1-\rho)^{3/2}\right)n^3 + o(n^3)
\le \left(\alpha^3+(\rho-\alpha^2) \sqrt{\rho+\alpha^2}\right) n^3 + o(n^3). 
\end{align*}

\textbf{Case 2. $\omega_1 \ge n/2$. } 

We choose $G \in \Phi(n, \rho, \alpha, \alpha)$  with $\omega(G) = \omega_2$. 
As in Case 1, let $G_1' = G[V_1]$, $G_2' = G[I\cup V_2]$, and $B_2$ be the bipartite graph induced by the edges between $V_1$ and $I\cup V_2$. 
Then $G_1'$ is complete, $G_2'$ is empty, and 
\begin{align}\label{EQ:Z_{1}(G)-B2-omega2}
Z_{1}(G) = \omega_2 (\omega_2 -1)^2 + 2(\omega_2 -1) \left(|G| - \binom{\omega_2}{2}\right) + Z_{1}(B_2). 
\end{align} 
Let 
\begin{align*}
\left(r,s,\ell,k,m\right) 
\coloneqq \left(\omega_2, n -\omega_2, \alpha n, \alpha n, |G|- \tbinom{\omega_2}{2}\right). 
\end{align*}
Since $\omega_2 \ge \omega_1 \ge n/2$, we have $r \ge s$. 
It follows from $\omega_2 \ge n/2$ and $\alpha \le 23/100$ that 
\begin{align*}
k + \ell 
= 2 \alpha n 
\le \frac{23}{50}n 
< \omega_2
= r. 
\end{align*}

By Theorem~\ref{Thm:biparite-k-ell-Z_{1}(G)}, we obtain that 
\begin{align}\label{EQ:alpha-Z_1-B_2-case2-from_bi}
 Z_{1}(B_2)
        \le 
        \begin{cases}
            Z_{1}\big(B_1(r,s,m,\ell,k) \big), &\quad\text{if}\quad m \le rk,  \\[0.5em] 
            Z_{1}\big(B(r,s,m) \big), &\quad\text{if}\quad m \ge r k. 
        \end{cases}
\end{align}

Let $G^{\ast}$ denote the graph obtained from $G$ by replacing $B_2$ with $B_1(r,s,m,\ell,k)$ or $B(r,s,m)$ (depending on whether $m \le rk$). 
Similar to the argument above, we have $G^{\ast} \in \Phi(n, \rho, \alpha, \alpha)$. 
Therefore, we have $|B_2| < (\alpha n+1) \omega_2$. 
Otherwise, $v_{\omega_2} v_{\omega_2 +1} \in G^{\ast}$, contradicting the maximality of $\omega_2$. 

\textbf{Subcase 2.1.} $m \ge rk$, that is, $|B_2| \ge \alpha n \cdot \omega_2$. 

In this subcase, we have $|B_2| = \alpha n \cdot \omega_2 + o(n^2)$. 
It follows that 
\begin{align*}
\rho \binom{n}{2} 
= |G| 
= \binom{\omega_2}{2} + |B_2|
= \frac{1}{2} \omega_2^2 + \alpha n \cdot \omega_2 + o(n^2), 
\end{align*}
which implies that 
\begin{align}\label{EQ:alpha-calculation-w_2-case2}
\omega_2
= \left(\sqrt{\rho + \alpha^2} - \alpha \right)n + o(n). 
\end{align}
Since $\alpha n \cdot \omega_2 \le |B_2| < (\alpha n+1) \omega_2$, by~\eqref{EQ:alpha-Z_1-B_2-case2-from_bi}, 
\begin{align*}
Z_{1}(B_2) 
\le Z_{1}\big(B(r,s,m) \big)
= Z_{1}\left(K_{\omega_2, \alpha n}\right) + o(n^3). 
\end{align*}
Substituting it into \eqref{EQ:Z_{1}(G)-B2-omega2} yields that 
\begin{align}\label{EQ:alpha-case2-complete_Z_{1}(G)_bound}
Z_{1}(G) 
\le \omega_2^3 + 2 \omega_2 |B_2| + Z_{1}(K_{\omega_2, \alpha n}) + o(n^3) 
= \omega_2^3 + 3 \alpha n \cdot \omega_2^2 + \alpha^2 n^2 \cdot \omega_2 + o(n^3). 
\end{align}
Combining~\eqref{EQ:alpha-calculation-w_2-case2} and~\eqref{EQ:alpha-case2-complete_Z_{1}(G)_bound}, we have 
\begin{align*}
Z_{1}(G) \le \left(\alpha^3+(\rho-\alpha^2)\sqrt{\rho+\alpha^2}\right) n^3 
+ o(n^3). 
\end{align*}

\textbf{Subcase 2.2. } $m \le rk$, that is, $|B_2| < \alpha n \cdot \omega_2$. 

By~\eqref{EQ:alpha-Z_1-B_2-case2-from_bi}, we have $Z_{1}(B_2) \le  Z_{1}\big(B_1(r,s,m,\ell,k) \big)$. 
According to the definition of $B_1(r,s,m,\ell,k)$, let $m = |B_2| = \alpha^2 n^2 + p(\omega_2 - \alpha n) +q$, where $0 \le q < \omega_2 - \alpha n$. 
Then, we have 
\begin{align*}
\rho \binom{n}{2} 
= |G| 
= \binom{\omega_2}{2} + \alpha^2 n^2 + p(\omega_2 - \alpha n) + o(n^2) 
= \frac{1}{2} \omega_2^2 + \alpha^2 n^2 + p(\omega_2 - \alpha n) + o(n^2), 
\end{align*}
and 
\begin{align*}
Z_{1}(G) 
&\le \omega_2 (\omega_2 -1)^2 + 2(\omega_2 -1) \left(|G| - \binom{\omega_2}{2}\right) + Z_{1}\big(B_1(r,s,m,\ell,k) \big) \\ 
&= \omega_2 ^3 + 2 \omega_2 \left(\alpha^2 n^2 + p(\omega_2 - \alpha n)\right)+p \omega_2^2 + (\alpha n -p) \alpha^2n^2 + \alpha^3 n^3 + (\omega_2 - \alpha n) p^2 + o(n^3) \\
& = p \omega_2^2 + (\alpha n - p) \alpha^2 n^2 + \alpha n (\alpha n+ \omega_2)^2 + (\omega_2 - \alpha n) (p + \omega_2)^2 + o(n^3). 
\end{align*}
Regarding this as a function of $p$ and applying Lemma~\ref{Lem:Appe-calculation_1} in Appendix (with $d = \rho/2$, $a = \alpha$, $y = \omega_2/n$ and $x = p/n$), we deduce that $Z_{1}(G)$ attains its maximum when $p = \alpha n$. 
Therefore, 
\begin{align*}
Z_{1}(G) 
\le \left(\alpha^3+(\rho-\alpha^2)\sqrt{\rho+\alpha^2}\right) n^3 + o(n^3).
\end{align*}

This completes the proof of Theorem~\ref{Thm:Max-S2-alpha-alpha}. 
\end{proof}

\section{Proof of Theorem~\ref{Thm:Max-S2-alpha-beta}}\label{Sec:note-general-graph-alpha-beta}

In this section, we present the proof of Theorem~\ref{Thm:Max-S2-alpha-beta}. 
The proof strategy follows a similar approach to Theorem~\ref{Thm:Max-S2-alpha-alpha}, although requiring more meticulous analysis and detailed calculations. 

\begin{proof}[Proof of Theorem~\ref{Thm:Max-S2-alpha-beta}]
Let $n$ be sufficiently large, $\rho \in \left[\frac{17}{25}, \frac{7}{10}\right]$ and $\alpha \in \left[\frac{1}{3}, \frac{2}{5}\right]$. 
Let $G$ be an $n$-vertex graph with edge density $\rho$, containing an independent set $I \subseteq V(G)$ with $|I| = \alpha n$ and $\delta_G(I) \ge \tfrac{1}{5}n$. 
By the lower bound given in Fact~\ref{FACT:lower-bound} and~\eqref{equ:Zagreb-index-S2}, it suffices to show that 
\begin{align*}
Z_{1}(G) 
\le \max\left\{2\rho -1 +(1-\rho)^{3/2}, \alpha^3+(\rho-\alpha^2) \sqrt{\rho+\alpha^2}\right\} n^3 + o(n^3). 
\end{align*}

Let $G \in \Phi(n, \rho, \alpha, 1/5)$, that is, $G$ is an extremal graph. 
Let $I = \{u_1,\ldots,u_{\alpha n}\}$ be the independent set of $G$, and let $V \coloneqq V(G) \setminus I = \{v_1,\ldots, v_{(1-\alpha)n}\}$. 
By relabeling the vertices, we may assume that 
\begin{align}\label{EQ:alpha-beta-degree-sequence}
d(u_1) \ge \cdots \ge d(u_{\alpha n}), 
\quad \text{and} \quad
d(v_1) \ge \cdots \ge d(v_{(1-\alpha)n}). 
\end{align}
It follows from $\delta_{G}(I) \ge \tfrac{1}{5}n$ that
\begin{align}\label{EQ:alpha-beta-degree>0.2n}
d(u_1) \ge \cdots \ge d(u_{\alpha n}) \ge \frac{1}{5}n.  
\end{align}

Similar tho proof of Claim~\ref{Claim:complete-bipartite-graph-bi} and Claim~\ref{Claim:assume-G-after-shifting}, using Fact~\ref{FACT:Z_{1}(G)-del-uv-add-xy} and assumption~\eqref{EQ:alpha-beta-degree-sequence}, we have the following claim. 

\begin{claim}\label{Claim:assume-G-after-shifting-n/5}
The following statements hold.  
\begin{enumerate}[label=(\roman*)]
\item \label{ITEM:alpha-beta-ui-vj-complete} If $u_iv_j \in G$, then  $u_iv_{j'} \in G$ for every $j' \in [j-1]$. 
\item \label{ITEM:alpha-beta-vi-vj-complete} If $v_iv_j\in G$, then $v_{i'}v_{j'} \in G$ for all $i' \le i$, $j' \le j$, and $i' \neq j'$.
\end{enumerate}
\end{claim}

It follows from Claim~\ref{Claim:assume-G-after-shifting-n/5} and~\eqref{EQ:alpha-beta-degree>0.2n} that 
\begin{align}\label{EQ:beta-has-complete-bipartite}
K[\{v_1, \ldots,v_{n/5}\}, \{u_1, \ldots,u_{\alpha n}\}] \subseteq G. 
\end{align}
Denote the largest index $i \in \left[(1-\alpha)n\right]$ such that $v_{i-1}v_i\in  G$ by $\omega \coloneqq \omega(G)$. 
Then the subgraph induced by $V_1 \coloneqq \{v_1,\ldots, v_{\omega}\}$ is complete, and the subgraph induced by $V_2 \coloneqq  \{v_{\omega+1}, \ldots, v_{(1-\alpha)n}\}$ is empty by the maximality of $\omega$ and Claim~\ref{Claim:assume-G-after-shifting-n/5}. 
It follows from 
\begin{align*}
|G| = \frac{1}{2} \rho n^2 
\le \binom{n}{2} - \binom{\alpha n}{2} -\binom{(1-\alpha)n -\omega}{2}
< \frac{1}{2}n^2 - \frac{1}{2}\alpha^2 n^2 -\frac{1}{2}\left((1-\alpha)n -\omega\right)^2
\end{align*}
that 
\begin{align}\label{EQ:bound-omega-n/5}
\omega 
> (1-\alpha)n - \left(1-\rho -\alpha^2 \right)^{1/2}n 
\ge (1-\alpha)n - \left(1-\frac{17}{25} -\alpha^2 \right)^{1/2}n 
\ge \frac{1}{5}n, 
\end{align}
where the last inequality follows from the fact that $\alpha \le \tfrac{2}{5}$. 

Using~\eqref{EQ:bound-omega-n/5} and the same argument as in the proof of Claim~\ref{CLAIM:alpha-B_0-empty} in Theorem~\ref{Thm:Max-S2-alpha-alpha}, we have the following Claim. 

\begin{claim}\label{CLAIM:beta-B_0-empty}
There exits a graph $G \in \Phi(n, \rho, \alpha, 1/5)$ such that $G[I \cup V_2]$ is empty. 
\end{claim}

\begin{proof}[Proof of Claim~\ref{CLAIM:beta-B_0-empty}]
Let $B_0 \coloneqq  G[I\cup V_2]$. 
Since $I$ is an independent set and $G[V_2]$ is empty by the definition of $\omega(G)$, $B_0$ is bipartite graph. 
We now show that there exits a graph $G \in \Phi(n, \rho, \alpha, 1/5)$ such that $|B_0| = 0$. 
Assume, for contradiction, that no such graph $G$ exists. 
We choose $G\in \Phi(n, \rho, \alpha, 1/5)$ such that $|B_0|$ is as small as possible. Then $|B_0|\ge 1$. 
By Claim~\ref{Claim:assume-G-after-shifting-n/5}, we have $u_1v_{\omega+1}\in G$. 
Let 
\begin{align*}
i_1 \coloneqq \max \left\{i \colon u_1v_i \in G\right\}, 
\quad \text{and} \quad
j_1\coloneqq \min\left\{j \colon v_{\omega +1}v_{j} \not\in G \right\}. 
\end{align*} 
Then $i_1 = d(u_1) \ge \omega + 1$ and $j_1 \le \omega$ by the choice of $\omega$. 
Moreover, $v_{j_1}v_j \not\in G$ for every $j \ge \omega+1$ by Claim~\ref{Claim:assume-G-after-shifting-n/5}. 
Let $G^{\ast}$ be the graph obtained from $G$ by removing the edge set $\{u_1 v_i \colon i \in [\omega+1,i_1]\}$ and then adding the set $\{v_{j_1} v_i \colon i \in [\omega+1,i_1]\}$. 
It is clear that $|G^{\ast}| = |G|$ and by~\eqref{EQ:bound-omega-n/5}, the minimum degree condition $\delta_{G^{\ast}}(I) \ge \tfrac{1}{5}n$ is preserved. 
Moreover, we have 
\begin{align*}
Z_{1}(G^{\ast}) - Z_{1}(G)
& = \left(d(u_1)-(i_1-\omega)\right)^2 + \left(d(v_{j_1})+(i_1-\omega)\right)^2 - d(u_1)^2 - d(v_{j_1})^2 \\
& = 2(i_1 - \omega)(d(v_{j_1}) - \omega) 
\ge 0. 
\end{align*}
Therefore, we have $G^{\ast} \in \Phi(n, \rho, \alpha, 1/5)$. 
This contradicts the minimality of $|B_0|$. 
\end{proof}

By Claim~\ref{CLAIM:beta-B_0-empty}, we may restrict our consideration to the graph $G \in \Phi(n, \rho, \alpha, 1/5)$ with $\left|G[I \cup V_2]\right| = 0$. 
We denote by $\omega_1$ and $\omega_2$ the minimum and maximum values of $\omega(G)$ among all graphs $G \in \Phi(n, \rho, \alpha, 1/5)$ with $\left|G[I\cup V_2]\right| = 0$, respectively. 
We distinguish the following two cases. 

\textbf{Case 1. $\omega_1 < n/2$. }

The argument is identical to that in \textbf{Case 1} of the proof of Theorem~\ref{Thm:Max-S2-alpha-alpha}, yielding 
\begin{align*}
Z_{1}(G) \le \left(2\rho -1 +(1-\rho)^{3/2}\right)n^3 + o(n^3). 
\end{align*}

\textbf{Case 2. $\omega_1 \ge n/2$. }

Choose a graph $G\in \Phi(n, \rho, \alpha, 1/5)$ with $\omega(G) = \omega_2$. 
It follows from $\omega_1 \ge n/2$ that $\omega_2 \ge n/2$. 
Let $G_1' = G[V_1]$, $G_2' = G[I\cup V_2]$, and $B_2$ be the bipartite graph induced by the edges between $V_1$ and $I \cup V_2$. 
Then $G_1'$ is complete, $G_2'$ is empty by Claim~\ref{CLAIM:beta-B_0-empty}, and 
\begin{align}\label{EQ:Z_{1}(G)-B2-omega2-n/5}
Z_{1}(G) 
= \omega_2 (\omega_2 -1)^2 + 2(\omega_2 -1) \left(|G| - \binom{\omega_2}{2}\right) + Z_{1}(B_2). 
\end{align}

The primary difficulty arises from the inapplicability of Theorem~\ref{Thm:biparite-k-ell-Z_{1}(G)} to $B_2$. 
To address this limitation, we introduce additional parameters and subsequently optimize $Z_{1}(B_2)$ over the feasible parameter space. 
To simplify the presentation, we relabel the vertices in $V_2 \cup I$ by $Z \coloneqq \left\{z_1, \ldots, z_{n-\omega_2}\right\}$ with $z_i = u_i$ for $i \in [\alpha n]$, and $z_{\alpha n+j} = v_{\omega_2+j}$ for $j \in [n-\alpha n-\omega_2]$. 
Using the same argument as Claim~\ref{Claim:complete-bipartite-graph-bi} in the proof of Theorem~\ref{Thm:biparite-ell-k-Z_{1}(G)}, we may further assume that 
\begin{align*}
\text{if } v_iz_j\in B_2, \text{ then } v_{i'}z_{j'} \in B_2 \text{ for all } i' \le i \text{ and } j' \le j. 
\end{align*}

It follows from~\eqref{EQ:beta-has-complete-bipartite} that 
\begin{align*}
K[\{v_1, \ldots,v_{n/5}\}, \{z_1, \ldots,z_{\alpha n}\}] \subseteq B_2. 
\end{align*}

Among all such $B_2$ that achieve the maximum value of $Z_{1}(B_2)$, that is, $B_2$ is a member of $\Phi\left(\omega_2,n-\omega_2, n/5, \alpha n, |G|-\tbinom{\omega_2}{2}\right)$, choose $n/5 \le i_0 \le \omega_2$ and $\alpha n \le j_0 \le n-\omega_2$ such that $\{v_1, \ldots, v_{i_0}\}$ and $\{z_1, \ldots, z_{j_0}\}$ form a complete bipartite graph in $B_2$ and $i_0 +j_0$ is maximized. 
We partition $B_2$ into the following four bipartite graphs: 
\begin{itemize}
\item $H_1 \coloneqq B_2\left[\{v_1,\ldots, v_{i_0}\}\cup \{z_1, \ldots, z_{j_0}\}\right]$, 

\item $H_2 \coloneqq B_2\left[\{v_1,\ldots, v_{i_0}\}\cup \{z_{j_0+1}, \ldots, z_{n-\omega_2}\}\right]$, 

\item $H_3 \coloneqq B_2\left[\{v_{i_0+1},\ldots, v_{\omega_2}\}\cup \{z_1, \ldots, z_{j_0}\}\right]$,

\item $H_4\coloneqq B_2\left[\{v_{i_0+1},\ldots, v_{\omega_2}\}\cup \{z_{j_0+1}, \ldots, z_{n-\omega_2}\}\right]$.  
\end{itemize}

Then $H_1$ is a complete bipartite graph and $H_4$ is empty. 
Moreover, we have 
\begin{align}\label{EQ:Z_{1}(B_2_example)-beta} 
Z_{1}(B_2) 
&= \sum_{i=1}^{\omega_2} d_{B_2}(v_i)^2 + \sum_{j=1}^{n-\omega_2} d_{B_2}(z_j)^2\notag \\
&= \sum_{i=1}^{i_0} d_{B_2}(v_i)^2 +\sum_{i=i_0+1}^{\omega_2} d_{B_2}(v_i)^2 + \sum_{j=1}^{j_0} d_{B_2}(z_j)^2+ \sum_{j=j_0+1}^{n-\omega_2} d_{B_2}(z_j)^2\notag \\
&= \sum_{i=1}^{i_0} \left(j_0+d_{H_2}(v_i)\right)^2+\sum_{i=i_0+1}^{\omega_2} d_{H_3}(v_i)^2+ \sum_{j=1}^{j_0} \left(i_0+d_{H_3}(z_j)\right)^2+ \sum_{j=j_0+1}^{n-\omega_2} d_{H_2}(z_j)^2\notag \\
& = i_0 \cdot j_0^2+ 2j_0 |H_2| + j_0 \cdot i_0^2+ 2i_0 |H_3|+Z_{1}(H_2) + Z_{1}(H_3). 
\end{align}

Considering the number of edges of $G$, 
we have 
\begin{align}\label{EQ:number-edges-G-H2-H3-n/5}
\frac{1}{2} \rho n^2
= |G| 
= \binom{\omega_2}{2} + |B_2|
= \binom{\omega_2}{2} + i_0 j_0 + |H_2| + |H_3|. 
\end{align}
Therefore, for fixed $i_0, j_0, |H_2|$ and $|H_3|$, combining~\eqref{EQ:Z_{1}(G)-B2-omega2-n/5} and~\eqref{EQ:Z_{1}(B_2_example)-beta}, it suffices to maximize $Z_{1}(H_2)$ and $Z_{1}(H_3)$. 

Note that $H_2$ is an $i_0$ by $n- \omega_2 -j_0$ bipartite graph with 
\begin{align*}
n- \omega_2 -j_0 
\le \frac{1}{2}n - j_0 
\le \frac{1}{2}n - \alpha n 
\le  \frac{1}{2}n - \frac{1}{3}n  
< \frac{1}{5}n
\le i_0. 
\end{align*}
By Theorem~\ref{Thm:AK78-bipartite}, we obtain that 
\begin{align}\label{EQ:beta-Z_1-H_2-graph-bound}
Z_{1}(H_2) 
\le Z_{1}\big(B(i_0,n- \omega_2 -j_0,|H_2|) \big). 
\end{align}
Let $\hat{B}_2$ denote the graph obtained from $B_2$ by replacing $H_2$ with $B(i_0,n- \omega_2 -j_0,|H_2|)$. 
Similar to the argument above, we have $\hat{B}_2 \in \Phi\left(\omega_2,n-\omega_2, n/5, \alpha n, |G|-\tbinom{\omega_2}{2}\right)$. 
Then we have 
\begin{align}\label{EQ:bound_H2}
|H_2| \le i_0 -1,    
\end{align}
since otherwise $v_{i_0}$ is adjacent to $z_{j_0+1}$ in $\hat{B}_2$, contradicting the maximality of $i_0 +j_0$. 

Recall that $H_3$ is a $j_0$ by $\omega_2 - i_0$ bipartite graph. 
If $j_0 \ge \omega_2 - i_0$, by the same reasoning as before, we have  
\begin{align}\label{EQ:bound_H3_1}
|H_3| \le j_0-1. 
\end{align}

Suppose that $j_0 < \omega_2 - i_0$. 
We claim that if $j_0 \ge \alpha n+1$, then 
\begin{align}\label{EQ:bound_H3_2}
|H_3| \le \omega_2 -i_0 -1. 
\end{align}

By contradiction, we assume that $|H_3| \ge \omega_2 -i_0$. 
Since $j_0 < \omega_2 - i_0$, by Theorem~\ref{Thm:AK78-bipartite}, we obtain that 
\begin{align*}
Z_{1}(H_3) 
\le Z_{1}\big(B(\omega_2 -i_0, j_0, |H_3|) \big). 
\end{align*}
Let $\hat{B}_2$ denote the graph obtained from $B_2$ by replacing $H_3$ with $B(\omega_2 -i_0, j_0, |H_3|)$. 
Similar to the argument above, we have $ \hat{B}_2 \in \Phi\left(\omega_2,n-\omega_2, n/5, \alpha n, |G|-\tbinom{\omega_2}{2}\right)$. 
Since $|H_3| \ge \omega_2 -i_0$, we have $z_1v_i \in \hat{B}_2$ for all $i_0 + 1 \le i \le \omega_2$. 
It follows from $j_0 \ge \alpha n+1$ that $z_1v_i, z_{\alpha n+1}v_i \in \hat{B}_2$ for all $i \le i_0$. 
Let $B^{\ast}_2$ denote the graph obtained from $\hat{B}_2$ by swapping the neighborhoods of $z_1$ and $z_{\alpha n+1}$. 
This results in $z_{\alpha n +1} v_{i} \in B^{\ast}_2$ for all $i \le \omega_2$. 
Let $G^{\ast}$ denote the graph obtained from $G$ by replacing $B_2$ with $B^{\ast}_2$. 
We have $v_{\omega_2} v_{\omega_2+1}  \in G^{\ast}$, contradicting the definition of $\omega_2$. 

Note that by the definition of $j_0$, we have $\alpha n \le j_0 \le n - \omega_2$. 
We first analyze the two boundary cases where $j_0 = \alpha n$ and $j_0 = n - \omega_2$ separately, and then demonstrate that $Z_{1}(B_2)$ achieves its maximum value in one of these two boundary cases. 

\textbf{Subcase 2.1. $j_0 = \alpha n$. } 

We divide this case into the following two subcases. 

\textbf{Subcase 2.1.1. $j_0 = \alpha n \ge \omega_2 - i_0$. }

It follows from~\eqref{EQ:bound_H2} and~\eqref{EQ:bound_H3_1} that $|H_2| \le i_0 -1$ and $|H_3| \le j_0 - 1 = \alpha n-1$. 
Set $x \coloneqq i_0 -n/5$. 
By~\eqref{EQ:number-edges-G-H2-H3-n/5}, 
we have 
\begin{align*}
\frac{1}{2} \rho n^2
= |G| 
= \frac{1}{2} \omega_2^2 + \frac{1}{5} \alpha n^2 + x \cdot \alpha n+ o(n^2). 
\end{align*}
Combining~\eqref{EQ:Z_{1}(G)-B2-omega2-n/5} and~\eqref{EQ:Z_{1}(B_2_example)-beta}, we have 
\begin{align*}
Z_{1}(G) 
& \le \omega_2^3 + 2\omega_2 \left(|G| - \binom{\omega_2}{2}\right) + i_0 \cdot j_0^2+ 2j_0 |H_2| + j_0 \cdot i_0^2+ 2i_0 |H_3|+Z_{1}(H_2) + Z_{1}(H_3)\\
&= \omega_2^3 + 2\omega_2 \left(\frac{1}{5} \alpha n^2 + x \cdot \alpha n\right) + i_0 \cdot \alpha^2 n^2 + i_0^2 \cdot \alpha n +o(n^3)\\
&= i_0(\omega_2 +\alpha n)^2 +(\omega_2 -i_0) \omega_2^2 + i_0^2 \cdot \alpha n +o(n^3)\\
&= \left(x+\frac{1}{5}n\right)(\omega_2 +\alpha n)^2 +\left(\omega_2 -x-\frac{1}{5}n\right) \omega_2^2 + \alpha n \cdot \left(x+\frac{1}{5}n\right)^2 +o(n^3).
\end{align*}

Note that $0 \le x \le \omega_2 -n/5$. 
Treating this as a function of $x$ and applying Lemma~\ref{Lem:calculation_2} in Appendix (with $d = \rho/2$, $a =\alpha$ and $y =\omega_2/n$), we find that $Z_{1}(G)$ attains its maximum at $x = \omega_2 - n/5$, and 
\begin{align*}
Z_{1}(G) 
\le \left(\alpha^3+(\rho-\alpha^2) \sqrt{\rho+\alpha^2}\right)n^3 + o(n^3).
\end{align*}

\textbf{Subcase 2.1.2. } $j_0 = \alpha n < \omega_2 -i_0$. 

Now consider $H_2$. 
Let $\hat{B}_2$ denote the graph obtained from $B_2$ by replacing $H_2$ with $B(i_0,n- \omega_2 -j_0,|H_2|)$. 
By~\eqref{EQ:Z_{1}(B_2_example)-beta} and~\eqref{EQ:beta-Z_1-H_2-graph-bound}, we have $Z_{1}(B_2) \le Z_{1}(\hat{B}_2)$. 
By Theorem~\ref{Thm:AK78-bipartite} and~\eqref{EQ:bound_H2}, we have $d_{\hat{B}_2}(z_j) = 0$ for all $j \ge \alpha n+2$, and $N_{\hat{B}_2}(z_{\alpha n+1}) = \left\{v_1, \ldots, v_{|H_2|}\right\}$ with $|H_2| \le i_0-1$. 
To handle this case, we adjust the partition of $\hat{B}_2$ according to the number of edges in $H_2$. 

Suppose that $|H_2| \le n/5$. 
Let  
\begin{itemize}
\item $\hat{H}_1 \coloneqq  \hat{B}_2\left[\left\{v_1,\ldots,v_{n/5}\right\} \cup \left\{z_1,\ldots,z_{\alpha n}\right\}\right]$, 

\item $\hat{H}_2\coloneqq \hat{B}_2\left[\left\{v_1,\ldots,v_{n/5}\right\}\cup\left\{z_{\alpha n+1}, \ldots, z_{n-\omega_2}\right\}\right]$, 

\item $\hat{H}_3\coloneqq \hat{B}_2\left[\left\{v_{n/5 +1}, \ldots, v_{\omega_2}\right\}\cup \left\{z_1, \ldots, z_{\alpha n}\right\}\right]$,  

\item $\hat{H}_4\coloneqq \hat{B}_2\left[\left\{v_{n/5 +1}, \ldots, v_{\omega_2}\right\}\cup \left\{z_{\alpha n+1}, \ldots, z_{n-\omega_2}\right\}\right]$.  
\end{itemize}
Then $\hat{H}_1$ is a complete bipartite graph, $\hat{H}_4$ is empty, and $|\hat{H}_2| = |H_2| \le n/5$. 
Similar to~\eqref{EQ:Z_{1}(B_2_example)-beta}, we have 
\begin{align}\label{EQ:Q-index-B2-H'-n/5}
Z_{1}(\hat{B}_2) = \frac{1}{5}n \cdot \alpha^2 n^2 + 2\alpha n |\hat{H}_2|+ \alpha n \cdot \left(\frac{1}{5}n\right)^2+ \frac{2}{5}n |\hat{H}_3|+Z_{1}(\hat{H}_2)+Z_{1}(\hat{H}_3). 
\end{align}

Consider the $\alpha n$ by $\omega_2 - n/5$ bipartite graph $\hat{H}_3$. 
Since $\alpha n < \omega_2 - i_0 \le \omega_2 - n/5$, by Theorem~\ref{Thm:AK78-bipartite}, we have $Z_{1}(\hat{H}_3) \le Z_{1}\big(B(\omega_2 - n/5, \alpha n, |\hat{H}_3|)\big)$. 
Express $|\hat{H}_3|= p \left(\omega_2 - \frac{1}{5}n\right)+q$, where $0 \le q < \omega_2 - \frac{1}{5}n$. 
Then 
\begin{align*}
\frac{1}{2} \rho n^2
= |G| 
= \binom{\omega_2}{2}+ \frac{1}{5} \alpha n^2 + |\hat{H}_2|+|\hat{H}_3|
= \frac{1}{2}\omega_2^2 + \frac{1}{5} \alpha n^2 + p \left(\omega_2 - \frac{1}{5}n\right) +o(n^2). 
\end{align*}

By~\eqref{EQ:Z_{1}(G)-B2-omega2-n/5},~\eqref{EQ:Q-index-B2-H'-n/5} and $Z_{1}(B_2) \le Z_{1}(\hat{B}_2)$, we have 
\begin{align*}
Z_{1}(G)
&\le \omega_2^3 + 2\omega_2 \left(|G| - \binom{\omega_2}{2}\right) + \frac{1}{5}n \cdot \alpha^2 n^2 + \alpha n \cdot \left(\frac{1}{5}n\right)^2+ \frac{2}{5}n |\hat{H}_3|+Z_{1}(\hat{H}_3)+o(n^3)\\
&= \omega_2^3 + 2\omega_2 \left(\frac{1}{5} \alpha n^2 + p\left(\omega_2 - \frac{1}{5}n\right)\right) + \frac{1}{5}n \cdot \alpha^2 n^2 + \alpha n \cdot \left(\frac{1}{5}n\right)^2\\
& \quad + \frac{2}{5}n \cdot p\left(\omega_2 - \frac{1}{5}n\right)+\left(\omega_2-\frac{1}{5}n\right)p^2+p\left(\omega_2 - \frac{1}{5}n\right)^2+o(n^3)\\
&= \frac{1}{5}n (\omega_2 +\alpha n)^2 + \left(\omega_2 - \frac{1}{5}n\right) (\omega_2 + p)^2 + p \cdot \omega_2^2 + \left(\alpha n - p\right) \left(\frac{1}{5}n\right)^2 + o(n^3).
\end{align*}
Note that $0 \le p \le \alpha n$. 
Treating this as a function of $p$ and applying Lemma~\ref{Lem:calculation_3} in Appendix (with $d =\rho/2$, $a =\alpha$, $y =\omega_2/n$ and $x= p/n$), we find that $Z_{1}(G)$ attains its maximum at $p = \alpha n$. 
Thus, 
\begin{align*}
Z_{1}(G) 
\le \left(\alpha^3+(\rho-\alpha^2) \sqrt{\rho+\alpha^2}\right)n^3 + o(n^3).
\end{align*}

Now suppose that $|H_2| \ge n/5$. 
Let 
\begin{itemize}
\item $H_1^{\ast} \coloneqq \hat{B}_2\left[\left\{v_1,\ldots,v_{n/5}\right\} \cup \left\{z_1,\ldots,z_{\alpha n+1}\right\}\right]$, 

\item $H_2^{\ast} \coloneqq \hat{B}_2\left[\left\{v_1,\ldots,v_{n/5}\right\} \cup \left\{z_{\alpha n+2}, \ldots, z_{n-\omega_2}\right\}\right]$, 

\item $H_3^{\ast} \coloneqq \hat{B}_2\left[\left\{v_{n/5 +1}, \ldots, v_{\omega_2}\right\} \cup \left\{z_1, \ldots, z_{\alpha n+1}\right\}\right]$, 

\item $H_4^{\ast} \coloneqq \hat{B}_2\left[\left\{v_{n/5 +1}, \ldots, v_{\omega_2}\right\} \cup \left\{z_{\alpha n+2}, \ldots, z_{n-\omega_2}\right\}\right]$. 
\end{itemize}
Then $H_1^{\ast}$ is a complete bipartite graph, and $H_2^{\ast}, H_4^{\ast}$ are empty. 
Similar to~\eqref{EQ:Z_{1}(B_2_example)-beta}, we have 
\begin{align*}
Z_{1}(\hat{B}_2)
= \frac{1}{5}n (\alpha n+1)^2 + (\alpha n+1) \cdot \left(\frac{1}{5}n\right)^2+\frac{2}{5}n |H_3^{\ast}|+Z_{1}(H_3^{\ast}). 
\end{align*}

Using Theorem~\ref{Thm:AK78-bipartite} and the fact that $\alpha n + 1 \le \omega_2 -i_0 \le \omega_2 - n/5$, we have $Z_{1}(H_3^{\ast}) \le Z_{1}\big(B(\omega_2 - n/5, \alpha n+1, |H_3^{\ast}|)\big)$. 
Express $|H_3^{\ast}|= p^{\ast} \left(\omega_2 - \frac{1}{5}n\right) + q^{\ast}$, where $0 \le q^{\ast} < \omega_2 - n/5$. 
Now, we have 
\begin{align*}
\frac{1}{2} \rho n^2
= |G| 
= \binom{\omega_2}{2} + \frac{1}{5}n \cdot \left(\alpha n+1\right) +|H_3^{\ast}|
= \frac{1}{2}\omega_2^2 + \frac{1}{5} \alpha n^2+ p^{\ast} \left(\omega_2 - \frac{1}{5}n\right)+o(n^2), 
\end{align*}
and 
\begin{align*}
Z_{1}(G)
& \le \omega_2^3 + 2\omega_2 \left(|G| - \binom{\omega_2}{2}\right) + \frac{1}{5}n \cdot \alpha^2 n^2 + \alpha n \cdot \left(\frac{1}{5}n\right)^2+ \frac{2}{5}n |H_3^{\ast}|+Z_{1}(H_3^{\ast})+o(n^3)\\
&= \frac{1}{5}n (\omega_2 +\alpha n)^2 + \left(\omega_2 - \frac{1}{5}n\right) (\omega_2 + p^{\ast})^2 + p^{\ast} \omega_2^2 + \left(\alpha n - p^{\ast}\right) \left(\frac{1}{5}n\right)^2 + o(n^3).
\end{align*}

Since $d_{\hat{B}_2}(z_{\alpha n+1}) = |H_2| \le i_0 -1 \le \omega_2 -1$, we have $|H_3^{\ast}| \le (\alpha n +1)\left(\omega_2 -\frac{1}{5}n\right) - 1$, and then $p^{\ast} \le \alpha n$. 
Similar to the argument above, treating $Z_{1}(G)$ as a function of $p^{\ast}$ and applying Lemma~\ref{Lem:calculation_3} in Appendix (with $d =\rho/2$, $a = \alpha$, $y =\omega_2/n$ and $x= p^{\ast}/n$), we find that $Z_{1}(G)$ attains its maximum at $p^{\ast} = \alpha n$. 
Therefore, we have
\begin{align*}
Z_{1}(G) 
\le \left(\alpha^3+(\rho-\alpha^2) \sqrt{\rho+\alpha^2}\right)n^3 + o(n^3).
\end{align*}

\textbf{Subcase 2.2. } $j_0 = n- \omega_2$. 

By the Subcase 2.1, we may assume $j_0 \ge \alpha n+1$. 

It follows from~\eqref{EQ:bound_H3_1} and~\eqref{EQ:bound_H3_2} that $|H_3| \le n$, which together with~\eqref{EQ:number-edges-G-H2-H3-n/5} and~\eqref{EQ:bound_H2} shows that 
\begin{align}\label{EQ:edge-case-j0=n-omega-n/5}
\frac{1}{2} \rho n^2 
= |G| 
= \frac{1}{2}\omega_2^2 + (n-\omega_2) i_0 + o(n^2). 
\end{align}

Combining~\eqref{EQ:Z_{1}(G)-B2-omega2-n/5} and~\eqref{EQ:Z_{1}(B_2_example)-beta}, we have 
\begin{align*}
Z_{1}(G) 
&\le \omega_2^3+ 2 \omega_2 \left(|G| - \binom{\omega_2}{2}\right) + i_0(n- \omega_2)^2+(n- \omega_2)i_0^2 + o(n^3)\\
& = \omega_2^3+2 i_0 \omega_2 (n-\omega_2) + i_0(n- \omega_2)^2+(n- \omega_2)i_0^2 + o(n^3)\\
&= i_0 n^2+(\omega_2-i_0)\omega_2^2+(n- \omega_2)i_0^2 +o(n^3). 
\end{align*}

Now we bound $i_0/n$. 
It follows from~\eqref{EQ:edge-case-j0=n-omega-n/5} that 
\begin{align}\label{EQ:i0=f(n-omega)-n/5}
i_0 
=\frac{\rho n^2 - \omega_2^2}{2(n-\omega_2)} - o(n)
= n- \frac{1}{2}\left(n -\omega_2 + \frac{n^2-\rho n^2}{n-\omega_2}\right) - o(n).
\end{align}

It follows from $n/2 \le \omega_2 \le (1- \alpha)n$ that 
\begin{align}\label{EQ:bound-n-omega-n/5}
\alpha n \le n -\omega_2 \le n/2. 
\end{align}
Moreover, under the assumption that $\frac{17}{25}\le \rho \le \frac{7}{10}$, the function $g(x) \coloneqq x+ \frac{n^2-\rho n^2}{x}$ is monotonically decreasing on the interval $[\alpha n, n/2]$. 
Thus,
\begin{align*}
\left(\frac{5}{2}-2\rho\right)n 
\le g(x) 
\le \left(\alpha+ \frac{1-\rho}{\alpha}\right)n, 
\end{align*}
for $x\in [\alpha n, n/2]$. 
This together with~\eqref{EQ:i0=f(n-omega)-n/5} and~\eqref{EQ:bound-n-omega-n/5} shows that
\begin{align*}
i_0 
\ge n- \frac{1}{2}\left(\alpha+ \frac{1-\rho}{\alpha}\right)n -o(n)
= \left(1- \frac{1}{2}\alpha -\frac{1-\rho}{2 \alpha}\right)n - o(n), 
\end{align*}
and 
\begin{align*}
i_0 
\le n -\frac{1}{2}\left(\frac{5}{2}-2\rho\right)n
= \left(\rho -\frac{1}{4}\right)n. 
\end{align*}

We maximize $Z_{1}(G)$ as a function of $i_0/n$ on the interval $\left[1- \frac{1}{2}\alpha -\frac{1-\rho}{2 \alpha},\rho-\frac{1}{4}\right]$. 
By Lemma~\ref{Lem:calculation_5} in Appendix (taking $d = \rho/2$, $a = \alpha$, $y =\omega_2/n$ and $x= i_0/n$), we have 
\begin{align*}
Z_{1}(G) 
\le \left(\alpha^3+(\rho-\alpha^2) \sqrt{\rho+\alpha^2}\right)n^3 + o(n^3).
\end{align*}

Now we consider the general case. 
Since the case $j_0 = \alpha n$ has been considered in Subcase~2.1, we may assume that $j_0 \ge \alpha n+1$. 
By~\eqref{EQ:bound_H2}, we have $|H_2| = o(n^2)$. 
By~\eqref{EQ:bound_H3_1} and~\eqref{EQ:bound_H3_2}, we have $|H_3| = o(n^2)$. 
Therefore, by~\eqref{EQ:Z_{1}(G)-B2-omega2-n/5},~\eqref{EQ:Z_{1}(B_2_example)-beta} and~\eqref{EQ:number-edges-G-H2-H3-n/5}, we have 
\begin{align}\label{EQ:number-edge-G-general-case-n/5}
\frac{1}{2} \rho n^2 
= |G| 
= \binom{\omega_2}{2} + i_0j_0 + o(n^2),    
\end{align}
and 
\begin{align}\label{EQ:Z_{1}(G)-general-case-n/5}
Z_{1}(G) 
= \omega_2 (\omega_2 -1)^2 + 2(\omega_2 -1) \left(|G| - \binom{\omega_2}{2}\right)+i_0j_0 (i_0+j_0) + o(n^3). 
\end{align}
When $\omega_2$ is fixed, $i_0j_0$ is determined by~\eqref{EQ:number-edge-G-general-case-n/5}. 
Moreover, for a fixed $C \coloneqq i_0j_0$, we have $i_0 + j_0 = C/{j_0}+j_0$ achieves its maximum when $j_0$ lies on the boundary of its domain. 
Therefore, for a fixed $\omega_2$, $Z_{1}(G)$ attains its maximum when $j_0$ is on the boundary of its domain. 
The constraints on $j_0$ are $ \alpha n+1 \le j_0 \le n -\omega_2$.
We can also bound $j_0$ using the fact $n/5 \le i_0 \le \omega_2$. However, it cannot be better. 

Note that $|B_2| < \left(\alpha n+1\right)\omega_2$, since otherwise, by Theorem~\ref{Thm:AK78-bipartite}, $Z_{1}(B_2)$ can be maximized such that $z_{\alpha n+1}v_{\omega_2} \in B_2$, which implies that $v_{\omega_2 +1}v_{\omega_2} \in G$, contradicting the definition of $\omega_2$. 
For the lower bound of $j_0$, we have $j_0 \ge \max\big\{\alpha n +1, \tfrac{i_0 j_0}{\omega_2} \big\}$. 
Comparing these two bounds, we have
\begin{align*}
\frac{i_0 j_0}{\omega_2} 
\le \frac{i_0j_0}{|B_2|} \left(\alpha n + 1\right)
\le \alpha n+1. 
\end{align*}

It follows from $n/2 \le \omega_2 \le \left(1-\alpha\right) n$ and $1/3 \le \alpha \le 2/5$ that $n/2 \le \omega_2 \le 2n/3$. 
Therefore, we have 
\begin{align}\label{EQ:beta-last-calculation}
\frac{\rho n^2 - \omega_2^2}{2(n-\omega_2)} 
\ge \frac{\frac{17}{25} n^2 - \omega_2^2}{2(n-\omega_2)}
\ge \frac{\frac{17}{25} n^2 - \left(\frac{2}{3}n\right)^2}{2\left(n-\frac{2}{3}n\right)}
= \frac{53}{150}n. 
\end{align}
For the upper bound of $j_0$, we have $j_0 \le \min\big\{n - \omega_2, \tfrac{i_0 j_0}{n/5}\big\}$. 
Using~\eqref{EQ:beta-last-calculation} to compare these two bounds, we have
\begin{align*}
\frac{i_0 j_0}{n/5} 
= \frac{|G|- \binom{\omega_2}{2}-o(n^2)}{n/5}
= 5 \times \frac{\rho n^2- \omega_2^2}{2n} - o(n)
\ge 5 \times \frac{53}{150} \left(n-\omega_2\right) - o(n)
\ge n - \omega_2. 
\end{align*}

We conclude that $\alpha n+1 \le j_0 \le n - \omega_2$. 
Then by~\eqref{EQ:Z_{1}(G)-general-case-n/5}, $Z_{1}(G)$ attains its maximum when $j_0 = n - \omega_2$ or $j_0 = \alpha n+1$. 
If $j_0 = n - \omega_2$, then we are done by Subcase 2.2. 
The case $j_0 = \alpha n+1$ can be handled similarly to Subcase 2.1, so we omit the detailed calculations. 
This completes the proof of Theorem~\ref{Thm:Max-S2-alpha-beta}. 
\end{proof}

\section*{Acknowledgements}
We would like to thank the referees for the careful reading of the manuscript and for providing many valuable suggestions. 

\begin{appendix}
\section{Appendix}
In this section, we establish  some computationally involved lemmas, verified by Mathematica. 
All Mathematica calculations are available at \url{https://github.com/xliu2022/xliu2022.github.io/blob/main/AKS2.nb}. 
Throughout this section, we assume that $d$ is a real number with $d \in \left[\tfrac{17}{50}, \tfrac{7}{20}\right]$. 

\begin{lemma}\label{LEM:alpha-inequality-star-two-upper-bounds}
    Suppose that $a \in \left[\tfrac{17}{100}, \tfrac{23}{100} \right]$. 
    Then we have 
    \begin{align*}
        a^3+(2d-a^2) \sqrt{2d+a^2} 
        > 4d -1 +(1-2d)^{3/2}.  
    \end{align*}
\end{lemma}

\begin{proof}[Proof of Lemma~\ref{LEM:alpha-inequality-star-two-upper-bounds}]
Using the fact that $d \in \left[\tfrac{17}{50}, \tfrac{7}{20}\right]$ and $a \in \left[\tfrac{17}{100}, \tfrac{23}{100} \right]$, Mathematica verifies that 
\begin{align*}
& \quad \, \, a^3+(2d-a^2) \sqrt{2d+a^2} -\left(4d -1 +(1-2d)^{3/2}\right) \\
& \ge  a^3+(2d-a^2) \sqrt{2d+a^2} -\left(4d -1 +(1-2d)^{3/2}\right) \at{a = \frac{23}{100}, d=\frac{17}{50}}\\
& = \frac{6271 \sqrt{7329}}{1000000} - \frac{347833}{1000000}-\frac{16\sqrt{2}}{125} 
> 0. 
\end{align*}
\end{proof}


\begin{lemma}\label{Lem:Appe-calculation_1}
Let $a \in \left[\tfrac{17}{100}, \tfrac{23}{100}\right]$, $x \in \left[0,a\right]$ and $y \in \left[0,1-a \right]$ be real numbers satisfying 
$ y^2/2 + xy + a^2 - a x = d$. 
Let 
\begin{align*}
f_1(x) & \coloneqq xy^2 + (a-x)a^2 + a(a+y)^2+(y-a)(x+y)^2. 
\end{align*}
Then we have $f_1(x) \le f_1(a) = a^3 + (2d-a^2)\sqrt{2d+a^2}$.
\end{lemma}

\begin{proof}[Proof of Lemma~\ref{Lem:Appe-calculation_1}]
Solving $ y^2/2 + xy + a^2 - a x = d$, we obtain that 
\begin{align*}
y = -x + \sqrt{x^2 + 2ax +2d-2a^2}. 
\end{align*}
Substituting this into $f_1$ yields 
\begin{align}\label{EQ:function-f_1(x)-explicit}
f_1(x) = x^3 + a x^2 +(2d-x^2) \sqrt{ x^2+2 a x + 2 d-2 a^2} - 3 a^2 x + 2 a^3.  
\end{align}
Taking the derivative of $f_1(x)$ with respect to $x$, we have 
\begin{align*}
\diff{f_1}{x} = 3x^2 + 2a x-3a^2 + \frac{(2d-x^2)(x+a)}{\sqrt{x^2+2 a x + 2 d-2 a^2}} -2x\sqrt{ x^2+2 a x + 2 d-2 a^2}. 
\end{align*}
Using the fact that $d \in \left[\tfrac{17}{50}, \tfrac{7}{20}\right]$ and  $a \in \left[\tfrac{17}{100}, \tfrac{23}{100}\right]$, we have $\diff{f_1}{x} \ge 0$ for all $x \in \left[0,a\right]$ by Mathematica. 
This means that $f_1(x)$ is increasing on $[0,a]$, and hence $f_1(x) \le f_1(a)$. 
It follows from~\eqref{EQ:function-f_1(x)-explicit} that $f_1(a) = a^3 + (2d-a^2)\sqrt{2d+a^2}$, completing the proof of Lemma~\ref{Lem:Appe-calculation_1}. 
\end{proof}


\begin{lemma}\label{Lem:calculation_2}
Let  $ a \in \left[\tfrac{1}{3},\tfrac{2}{5}\right]$, $x \in \left[0,y-\tfrac{1}{5}\right]$, and $y \in \left[0,1-a \right]$ be real numbers satisfying $y^2/2 + ax+ a/5 = d$. 
Let 
\begin{align*}
f_2(x) 
&\coloneqq \left(x+\frac{1}{5}\right)(y+a)^2+\left(y-x-\frac{1}{5}\right)y^2+a\left(x+\frac{1}{5}\right)^2. 
\end{align*}
Then we have $f_2(x) \le a^3 +(2d-a^2)\sqrt{2d+a^2}$. 
\end{lemma}

\begin{proof}[Proof of Lemma~\ref{Lem:calculation_2}]
Solving $y^2/2+ ax+ a/5=d$, we obtain that 
\begin{align}\label{EQ:expression-y-f2}
y = \sqrt{2 \left(d-ax-\frac{1}{5}a\right)}. 
\end{align}
Substituting this into $f_2$ yields that 
\begin{align*}
f_2(x) = a\left(x+\frac{1}{5}\right)\left(x+a+\frac{1}{5}\right)+ 2d \sqrt{2 \left(d-ax-a/5\right)}.
\end{align*}
Thus, 
\begin{align*}
\diff{f_2}{x} = a \left(x+\frac{1}{5}\right)+ a \left(x+a+\frac{1}{5}\right)- ad \sqrt{\frac{2}{d-ax-a/5}},
\end{align*}
and then
\begin{align*}
\diff[2]{f_2}{x} = 2a- \frac{\sqrt{2}a^2 d}{2\left(d-ax-a/5\right)^{3/2}}. 
\end{align*}

By~\eqref{EQ:expression-y-f2} and $0 \le x \le y - 1/5$, we have 
\begin{align*}
0 \le x \le \theta(a,d) \coloneqq \sqrt{a^2+2d}-a -\frac{1}{5}.
\end{align*}
Now we bound $\theta$. 
It follows from $d \in \left[\tfrac{17}{50}, \tfrac{7}{20}\right]$ and $a \in \left[\tfrac{1}{3},\tfrac{2}{5}\right]$ that 
\begin{align*}
\theta(a,d) 
\ge \sqrt{a^2+2 \times \frac{17}{50}}-a -\frac{1}{5}
\ge \sqrt{\left(\frac{2}{5}\right)^2+2 \times \frac{17}{50}}-\frac{2}{5} -\frac{1}{5}
> \frac{3}{10}, 
\end{align*}
and 
\begin{align*}
\theta(a,d)
\le \sqrt{a^2+2 \times \frac{7}{20}} - a -\frac{1}{5}
\le \sqrt{\left(\frac{1}{3}\right)^2+2 \times \frac{7}{20}}-\frac{1}{3} -\frac{1}{5}< \frac{37}{100}. 
\end{align*}

If $0\le x \le \tfrac{3}{10}$, then $\diff[2]{f_2}{x} \ge 0$ by Mathematica. 
This means that $f_2(x)$ is convex, and hence,   
\begin{align}\label{EQ:Cal_f3_first_convex}
f_2(x) 
\le \max\left\{f_2(0),f_2(3/10)\right\}.
\end{align}

Suppose that $\frac{3}{10}\le x \le \frac{37}{100}$. 
It is easy to verify that $\diff{f_2}{x} \ge 0$ by Mathematica. 
Thus, $f_2(x) \le f_2\big(\theta(a,d)\big)$ for $\tfrac{3}{10} \le x \le \theta(a,d)$. 
This together with~\eqref{EQ:Cal_f3_first_convex} yields that 
\begin{align*}
f_2(x) 
&\le \max\left\{f_2(0), f_2\left(\sqrt{a^2+2d}-a -1/5\right)\right\}\\
&= \max \left\{\frac{1}{5}a^2+\frac{1}{25}a+2d\sqrt{2\left(d-\frac{1}{5}a\right)},a^3 +(2d-a^2)\sqrt{2d+a^2}\right\}. 
\end{align*}
Let 
\begin{align*}
g(a,d) \coloneqq a^3 +(2d-a^2)\sqrt{2d+a^2}-\left(\frac{1}{5}a^2+\frac{1}{25}a+2d\sqrt{2\left(d-\frac{1}{5}a\right)}\right). 
\end{align*}
Using the fact that $d \in \left[\tfrac{17}{50}, \tfrac{7}{20}\right]$ and $a \in \left[\tfrac{1}{3},\tfrac{2}{5}\right]$, we know $g(a,d) > 0$ by Mathematica, and then 
\begin{align*}
f_2(x) \le a^3 +(2d-a^2)\sqrt{2d+a^2}.
\end{align*}
\end{proof}


\begin{lemma}\label{Lem:calculation_3}
Let $a \in \left[\tfrac{1}{3},\tfrac{2}{5}\right]$, $x \in \left[0,a\right]$, and $y \in \left[0,1-a\right]$ be real numbers satisfying $y^2/2 + xy + a/5-x/5 = d$. 
Let
\begin{align*}
f_3(x) 
\coloneqq \frac{1}{5}(y+a)^2 + \left(y-\frac{1}{5}\right)(x+y)^2 + x y^2+ \left(a-x\right)\left(\frac{1}{5}\right)^2.
\end{align*}
Then we have $f_3(x) \le f_3(a) = a^3 + (2d - a^2)\sqrt{2d+a^2}$. 
\end{lemma}

\begin{proof}[Proof of Lemma~\ref{Lem:calculation_3}]
Solving $y^2/2 + xy + a/5-x/5 = d$, we obtain that 
\begin{align*}
y = \frac{1}{5} \sqrt{25x^2+10x+50d-10a}-x. 
\end{align*}
Therefore, substituting this into $f_3$, we have 
\begin{align*}
f_3(x) 
=\frac{1}{25}\left(25x^3+5x^2-10ax-x+5a^2+a-\left(5x^2-10d\right)\sqrt{25x^2+10x+50d-10a}\right). 
\end{align*}
Then we have 
\begin{align*}
\diff{f_3}{x} 
= \frac{1}{25}\Bigg(&75x^2+10x-10a-1-10x \sqrt{25x^2+10x+50d-10a}\\
&-\frac{(50x+10)(5x^2 -10d)}{2\sqrt{25x^2+10x+50d-10a}}\Bigg). 
\end{align*}

Note that using Mathematica, we have
\begin{align*}
\min \left\{f_3(a)-f_3(x) \colon x \in \left[0,\frac{1}{4}\right], a \in \left[\frac{1}{3},\frac{2}{5}\right], \text{~and~} d \in \left[\frac{17}{50},\frac{7}{20}\right] \right\}
> 0.002232. 
\end{align*}
Therefore, we have $f_3(x) < f_3(a)$ for $x \in \left[0,\tfrac{1}{4}\right]$. 

When $\tfrac{1}{4} \le x \le \tfrac{2}{5}$, one can verify that $\diff{f_3}{x} \ge 0$ by Mathematica, which means that $f_3(x)$ is increasing on $\left[\tfrac{1}{4},\tfrac{2}{5}\right]$. 
Since $\tfrac{1}{3} \le a \le \tfrac{2}{5}$, we have $f_3(x) \le f_3(a)$ for $x \in \left[\tfrac{1}{4},a\right]$. 
This completes the proof of Lemma~\ref{Lem:calculation_3}.  
\end{proof}


\begin{lemma}\label{Lem:calculation_5}
Let $ a \in \left[\tfrac{1}{3},\tfrac{2}{5}\right]$, $x\in \left[1- \tfrac{1}{2}a -\tfrac{1-2d}{2a},2d - \tfrac{1}{4}\right]$ and $y \in \left[0,1-a\right]$ be real numbers satisfying $y \ge x$ and $y^2/2 + (1-y)x = d$. 
Let 
\begin{align*}
f_4(x) \coloneqq x+(y-x)y^2+(1-y)x^2. 
\end{align*}
Then we have $f_4(x) < a^3+(2d-a^2) \sqrt{2d+a^2}$. 
\end{lemma}

\begin{proof}[Proof of Lemma~\ref{Lem:calculation_5}]
Solving $y^2/2 + (1-y)x = d$, we have 
\begin{align*}
y = x- \sqrt{x^2-2x+2d} \text{~or~} y = x+ \sqrt{x^2-2x+2d}. 
\end{align*}
It follows from $y \ge x$ that $y= x+ \sqrt{x^2-2x+2d}$ and then 
\begin{align*}
f_4(x) 
= x+(y-x)y^2+(1-y)x^2 
= x^3-3x^2+4dx+x+\left(x^2-2x+2d\right)^{3/2}.  
\end{align*}
Note that 
\begin{align*}
\diff{f_4}{x}
= 3x^2-6x +4d+1+3(x-1) \sqrt{x^2-2x+2d},
\end{align*}
and 
\begin{align*}
\diff[2]{f_4}{x}
= 6x-6+ \frac{3(x-1)^2}{\sqrt{x^2-2x+2d}}+3\sqrt{x^2-2x+2d}. 
\end{align*}

Since $d \in \left[\tfrac{17}{50}, \tfrac{7}{20}\right]$ and $a \in \left[\tfrac{1}{3},\tfrac{2}{5}\right]$, we have 
\begin{align*}
x 
\ge 1- \frac{1}{2}a -\frac{1-2d}{2a} 
\ge 1- \frac{1}{2} \times \frac{1}{3} -\frac{1-2\times \frac{17}{50}}{2\times \frac{1}{3}}
= \frac{53}{150}
> \frac{1}{3}, 
\end{align*}
and 
\begin{align*}
x
\le 2d - \frac{1}{4}
\le 2 \times \frac{7}{20} - \frac{1}{4}
= \frac{9}{20}. 
\end{align*}
Since $d \in \left[\tfrac{17}{50}, \tfrac{7}{20}\right]$, one can verify that $\diff[2]{f_4}{x} \ge 0$ on $x \in \left[\tfrac{1}{3},\tfrac{9}{20}\right]$ by Mathematica. 
Thus, 
\begin{align*}
f_4(x) 
\le \max \left\{f_4(1-a/2-(1-2d)/2a),f_4(2d-1/4)\right\} 
< a^3+\left(2d - a^2\right)\sqrt{2d+a^2}, 
\end{align*}
where the last inequality also follows from Mathematica and the fact that $\tfrac{17}{50} \le d \le \tfrac{7}{20}$ and $\tfrac{1}{3} \le a \le \tfrac{2}{5}$. 
This completes the proof of Lemma~\ref{Lem:calculation_5}. 
\end{proof}

\end{appendix}
\end{document}